\documentclass{amsart}
\usepackage[british]{babel}
\usepackage{geometry,amsthm,amsmath,amssymb,graphicx,enumerate,latexsym,calc,capt-of,xcolor,ifthen}
\usepackage{hyperref}
\usepackage{subfigure}
\catcode`\<=\active \def<{
	\fontencoding{T1}\selectfont\symbol{60}\fontencoding{\encodingdefault}}

\newcommand{\nin}{\not\in}
\newcommand{\nobracket}{}
\newcommand{\textdots}{...}
\newcommand{\tmaffiliation}[1]{\\ #1}

\newcommand{\tmop}[1]{\ensuremath{\operatorname{#1}}}
\newcommand{\tmstrong}[1]{\textbf{#1}}

\newcommand{\tmtextit}[1]{\text{{\itshape{#1}}}}

\newtheorem{assumption}{Assumption}

\newcommand{\tmfloatcontents}{}
\newlength{\tmfloatwidth}
\newcommand{\tmfloat}[5]{
	\renewcommand{\tmfloatcontents}{#4}
	\setlength{\tmfloatwidth}{\widthof{\tmfloatcontents}+1in}
	\ifthenelse{\equal{#2}{small}}
	{\setlength{\tmfloatwidth}{0.45\linewidth}}
	{\setlength{\tmfloatwidth}{\linewidth}}
	\begin{minipage}[#1]{\tmfloatwidth}
		\begin{center}
			\tmfloatcontents
			\captionof{#3}{#5}
		\end{center}
\end{minipage}}
\numberwithin{equation}{section}
\numberwithin{figure}{section}
\theoremstyle{plain}
\newtheorem{thm}{Theorem}[section]
\theoremstyle{remark}

\theoremstyle{plain}

\theoremstyle{plain}

\newcounter{casectr}

\theoremstyle{remark}

\theoremstyle{remark}

\theoremstyle{definition}

\theoremstyle{plain}

\theoremstyle{plain}
\newtheorem{cor}[thm]{Corollary}
\theoremstyle{plain}

\theoremstyle{definition}

\theoremstyle{definition}
\newtheorem{defn}[thm]{Definition}
\theoremstyle{definition}
\newtheorem{example}[thm]{Example}
\theoremstyle{definition}

\theoremstyle{plain}

\theoremstyle{plain}
\newtheorem{lem}[thm]{Lemma}
\theoremstyle{remark}
\newtheorem{notation}[thm]{Notation}
\theoremstyle{definition}

\theoremstyle{plain}
\newtheorem{prop}[thm]{Proposition}
\theoremstyle{remark}
\newtheorem{rem}[thm]{Remark}
\theoremstyle{remark}

\theoremstyle{plain}

\theoremstyle{remark}
\newtheorem*{rem*}{Remark}

\newcommand{\lmd}{\lambda}

\newcommand{\R}{\mathbb{R}}

\newcommand{\eps}{\varepsilon}

\newcommand{\llangle}{\left\langle}
\newcommand{\rrangle}{\right\rangle}

\newcommand{\ph}{\varphi}
\newcommand{\cl}{\mathcal}
\newcommand{\bb}{\mathbb}

\newcommand{\lf}{\left}
\newcommand{\rh}{\right}

\newcommand{{\HH}}{\mathbb{H}}

\begin{document}
	
	\title[CONCENTRATION OF INVARIANT MEASURES FOR S.D.S.]{The concentration of zero-noise limits of invariant measures for stochastic dynamical systems}

	\author[Z. Dong, F. Gu, L. Li]{
		Zhao Dong$^{1,2}$, Fan Gu$^{1,2,*}$, Liang Li$^{3}$
	\address{$*$ Corresponding Author}
	\email {gufan@amss.ac.cn}
	\tmaffiliation{\tiny
		1 Academy of Mathematics and Systems Science, Chinese Academy
			of Sciences, Beijing, 100190, China\\
			2 School of Mathematical Sciences, University of Chinese Academy of
			Sciences, Beijing, 100049, China\\
			3 College of Mathematics and Physics, Beijing University of Chemical Technology, Beijing,
						100029, China}}
	\maketitle
	
		\begin{abstract}
		In this paper, we study concentration phenomena of zero-noise limits of invariant measures for stochastic differential equations defined on $\mathbb{R}^d$ with locally Lipschitz continuous coefficients and more than one ergodic state. Under some dissipative conditions, by using Lyapunov-like functions and large deviations methods, we estimate the invariant measures in neighborhoods of stable sets, neighborhoods of unstable sets and their complement, respectively. Our result illustrates that invariant measures concentrate on the intersection of stable sets where a cost functional $W(K_i)$ is minimized and the Birkhoff center of the corresponding deterministic systems as noise tends down to zero. Furthermore, we prove the large deviations principle of invariant measures. At the end of this paper, we provide some explicit examples and their numerical simulations.  
	\end{abstract}

  \smallskip

  \vskip 0.5cm


  \vskip 0.3cm


	\tableofcontents

	\section{Induction}
	
	Let $b (x) : \mathbb{R}^d \rightarrow \mathbb{R}^d$ and
	$\sigma (x) : \mathbb{R}^d \rightarrow \mathbb{R}^{d \times d}$ be two locally
	Lipschitz continuous functions. We consider the ordinary differential equation
	\begin{equation}\label{1.1}
		\left\{ \begin{array}{ccl}
			d X_t & = & b (X_t) d t,\\
			X_0 & = & x_0,
		\end{array} \right. 
	\end{equation}
	where $x_0 \in \mathbb{R}^d$. We also consider the corresponding stochastic differential equation defined on $(\Omega, \{ \mathcal{F}_t \}_{t \geq 0},
	\mathcal{F}, P)$
	\begin{equation}\label{1.2}
		\left\{ \begin{array}{ccl}
			d X^{\varepsilon}_t &=& b (X^{\varepsilon}_t) d t + \varepsilon \sigma
			(X^{\varepsilon}_t) d W_t,\\
			X^{\varepsilon}_0 &=& x_0.
		\end{array} \right. 
	\end{equation}
	Here $x_0 \in \mathbb{R}^d$, $W_t$ is a standard $d$-dimension Brownian motion and the filtration $ \{ \mathcal{F}_t \}_{t \geq 0} $ satisfies the usual condition.
	
	For the strong solutions $X^\eps$ of equation  \eqref{1.2}, we use $P^{\varepsilon}_x := P \circ (X_{\cdot}^{\varepsilon})^{- 1}$
	to denote probability measures on the trajectory space. We denote invariant measures of \eqref{1.2} by $\mu_{\varepsilon}$ for any $\varepsilon\in(0, +\infty)$. If $\{\mu_\eps\}$ or its subsequence has a weak limit as $\varepsilon \downarrow 0$, then we denote the limit by $\mu$. 
	
	There have been a lot of literatures studying stochastic dynamical systems like \eqref{1.2}. Among them there are two important kinds of properties studied extensively. The first kind is to study exit problems, which includes the exit time and the exit location. Existing works about these problems can be devided into works for stable set networks and for heteroclinic networks, which can be found in \cite{peano} \cite{fw} \cite{yb11} \cite{rl} \cite{rt} and references therein. The second kind is to study properties of invariant measures $\{\mu_{\varepsilon}\} $. In this paper, we study concentration phenomena and the large deviations principle of $\{\mu_{\varepsilon}\} $ as $\varepsilon \downarrow 0$.
	
	There are many literatures devoted to these important problems. For example, in one dimension case, by using explicit
	solutions of stationary Fokker-Planck equations of \eqref{1.2}, \cite{1d} shows $\mu$ supports on points attaining the minimum of the energy function.
	By using Lyapunov-like functions and stationary Fokker-Planck equations, \cite{me} gets an estimate of the decay rate of
	$\{\mu_{\varepsilon}\}$ as $\varepsilon \downarrow 0$ for \eqref{1.2} defined on $\mathbb{R}^d$. \cite{chen} proves $\mu$ is also an invariant measure of \eqref{1.1}. 
	Thus, by using results in \cite{me} and \cite{mane}, for stochastic dynamical systems driven by L\'evy noise, \cite{chen} proves that $\mu$ supports on the Birkhoff center of \eqref{1.1} except repelling sets.
	\cite{fw} considers \eqref{1.2} defined on a compact space with continuous
	coefficients. By large deviations methods, \cite{fw} gets an estimate for $\{\mu_\eps\}$ in any sufficiently small neighborhood of equivalent sets. There are also some relative works on SPDEs such as \cite{li} \cite{sdpe17} \cite{spderd} \cite{sr} \cite{flan} \cite{ex1} \cite{ex2} \cite{ex3}.
	
	In general, stationary Fokker-Planck equations of \eqref{1.2} can be set on
	non-compact spaces with very general coefficients. Denote $\{\varrho_{\varepsilon}\}$ as solutions of stationary Fokker-Planck equations of \eqref{1.2}.
	Under some broad conditions, $\{\varrho_{\varepsilon}\}$ are
	density functions of $\{\mu_{\varepsilon}\}$. To study concentration phenomena of $\{\mu_{\varepsilon}\}$, one may need some asymptotic properties of $\{
	\varrho_{\varepsilon} \}$, which can be proved by using either explicit
	expressions of $\{\varrho_{\varepsilon}\}$ or decay properties of $\{
	\varrho_{\varepsilon} \}$, both rely heavily on uniformly decay properties
	of Lyapunov-like functions. Unfortunately, the
	Lyapunov-like functions can not have a uniformly decay property near the
	stationary sets of \eqref{1.1}. Therefore, for \eqref{1.1} with more than one ergodic state, this method can not go further to analyse measures of saddle points or stable sets under $\mu$. For the large deviations method provided in
	{\cite{fw}}, it relies on the compactness of space and the
	boundness of coefficients, which are necessary for the large deviations property and some essential estimates.
	
	In this paper, our main result shows that under conditions of Proposition \ref{prop:LDP} and
	Assumption \ref{mainassumption}, for \eqref{1.2} defined on $\mathbb{R}^d$ with locally Lipschitz continuous coefficients, $\{\mu_{\varepsilon}\}$ concentrates on the intersection of stable equivalent sets where a cost functional $ W(K_i)$ is minimized and the Birkhoff center of \eqref{1.1} as $\varepsilon \downarrow 0$. Furthermore, we show the large deviations principle of $\{\mu_{\varepsilon}\}$ and give the action function of $\{\mu_{\varepsilon}\}$. To achieve these goals, we mainly do following three aspects of work. First, by Lyapunov-like functions, we get the uniform large deviation property for $\{P_x^{\varepsilon}\}$ with respect to the initial point in any compact set. Second, we devide $\mathbb{R}^d$ into neighborhoods of stable equivalent sets, neighborhoods of unstable equivalent sets and the complement of neighborhoods of stable equivalent sets and unstable equivalent sets. By using the large deviations method and dissipative properties of \eqref{1.2} in a neighborhood of the stable equivalent sets and a domain outside of a compact set, we establish an estimate of $\{\mu_{\varepsilon}\}$ in the above three kinds of sets. Finally, we prove $\{\mu_{\varepsilon}\}$ satisfies the large deviations principle and give its action function, which shows the convergence rate of $\{\mu_\eps\}$.
	
	Our results show differences between the support of $ \mu$ and $\omega$-limit sets of \eqref{1.1}. These imply that long time behaviors of \eqref{1.1} and \eqref{1.2} are essentially different. Thus, under inevitable perturbations in real physical phenomena, in a long time observation, we will only be able to see the states that $\mu$ supports on.
	
	This paper is organized as follows: In section 2, for \eqref{1.2}
	with locally Lipschitz continuous coefficients in $\mathbb{R}^d$, we prove that under conditions of Proposition \ref{prop:LDP}, strong solutions of \eqref{1.2} satisfy the Freidlin-Wentzell type
	large deviations principle. In this section, the explicit action functional is also given. In section 3, we discuss connections between stabilities in the sense of quasi-potential and in the sense of deterministic dynamical systems. In section 4, under conditions of Proposition \ref{prop:LDP} and Assumption \ref{mainassumption}, we give an expression of $\mu_{\varepsilon}$ and an estimate of $\mu_{\varepsilon}$ in different domains. In section 5, we give some large deviations properties of $\{\mu_{\varepsilon}\}$. In section 6, we provide some examples with numerical simulations. We analyse these examples from both theoretical and numerical perspectives.
	
	\medskip
	
	\section{Large deviations of $ \{P_x^\varepsilon\} $}
	
	Let us first introduce some notations.
	\begin{notation}
		For $0 \leq T_1 \leq T_2 < + \infty$ and $x\in \R^d$, we denote
		$$C_x ([T_1, T_2] ; \mathbb{R}^d) = \lf\{\ph\in C([T_1, T_2] ; \mathbb{R}^d):\ph_{T_1}=x\rh\}.$$
		Similarly, for any $D \subset \mathbb{R}^d $, we denote $C_x ([T_1, T_2] ; D)$ as the space of continuous functions in $D$ from $T_1$ to $T_2$ beginning at $x$.
		
		Let $ C_x ([T_1, T_2] ;
		\mathbb{R}^d)$ be the space endowed with the metric
		\[\rho_{T_1 T_2} (\varphi, \psi) = \max_{t \in [T_1, T_2]} | \varphi_t -\psi_t |,\]
		Where $\varphi, \psi $ belong to $ C_x ([T_1, T_2] ;
		\mathbb{R}^d)$.
		
		For the sake of simplicity, we denote
		\[AC_x ([T_1, T_2]; \mathbb{R}^d) =\lf\{\varphi \in C_x ([T_1, T_2] ; \mathbb{R}^d): \ph \textrm{ is absolutely continuous}\rh\},\]
		and the Cameron Martin space
		\[H_0 ([T_1, T_2] ; \mathbb{R}^d) = \left\{\varphi \in A C_0 ([T_1, T_2]; \mathbb{R}^d):
		\int^{T_2}_{T_1} | \dot{\varphi}_s |^2 d s < + \infty \right\}\]
		with norm
		$$\| \varphi  \|_1 := \lf(\int^{T_2}_{T_1} | \dot{\varphi}_s |^2 d s\rh)^\frac{1}{2}.$$
		We also  denote
		$$H_x ([T_1, T_2] ; \mathbb{R}^d) =\lf\{\varphi:\varphi -x \in H_0 ([T_1, T_2] ; \mathbb{R}^d)\rh\}.$$

		We will use $X(x)$ to denote the solution of equation \eqref{1.1} with initial point $x\in\R^d$ and $X^\eps(x)$ to denote the strong solution of equation \eqref{1.2} with initial point $x\in\R^d$. And we denote the first entrance time for $O\subset \R^d$ as:
		\[ \tau_{O,x}=\inf\{t\ge0:X_t(x)\in O\}, \quad\tau_{O,x}^\eps=\inf\{t\ge0:X_t^\eps(x)\in O\}.\]
		
		We use $B_a(M):=\{x\in\R^d:|x-a|<M\}$ to represent the ball of radius $M>0$ and centered at $a\in\mathbb{R}^d$.
		
	\end{notation}
	
	The Freidlin-Wentzell large deviations principle in {\cite{fw}} can be used to deal
	with problems about limiting properties of $\{ P^{\varepsilon}_x \}$ as
	$\varepsilon \downarrow 0$ for bounded and uniformly continuous coefficients		$b(x)$ and $\sigma (x)$. Under those conditions, the LDP is uniformly with
	respect to the initial point $x \in \mathbb{R}^d$. {\cite{gfw}} and	references therein show that for $b (x)$ and $\sigma (x)$ satisfying local Lipschitz and linear growth conditions, $\{ P^{\varepsilon}_x \}$ also has
	Freidlin-Wentzell type LDP uniformly with respect to the initial point in any
	compact set belong to $\mathbb{R}^d$.
	
	In this section, we extend the LDP
	results in {\cite{gfw}} by Lyapunov-like function instead
	of the linear growth condition. The linear growth condition in {\cite{gfw}} is
	used to ensure that the uniqueness of solutions of \eqref{1.2} and \eqref{aum1} and some bounded estimates of them. This change of condition allows our model to
	be applied to some meaningful examples such as stochastic Duffing equation and Bernoulli equation.
	
	To investigate the LDP, let us introduce the action functional and the level set.
	Let $\sigma$ be an invertible matrix in $\mathbb{R}^{d\times d}$. 
	For $0 \leq T_1 \leq T_2 $ and $\varphi \in C_x ([T_1, T_2] ; \mathbb{R}^d)$, we set
	\begin{equation}\label{af1}
		S_{T_1 T_2} (\varphi) = \left\{ \begin{array}{ll}
			\frac{1}{2} \int^{T_2}_{T_1} | (\dot{\varphi}_t - b (\varphi_t))^T(\sigma\sigma^T)^{-1} (\varphi_t)
			(\dot{\varphi}_t - b (\varphi_t)) |^2 d t, &\varphi \in A C_x ([T_1,
			T_2] ; \mathbb{R}^d),\\ \\
			+ \infty,& \tmop{otherwise} .
		\end{array} \right.
	\end{equation}
	and
	\[ \Phi_x (s) = \{ \varphi \in C_x ([T_1, T_2] ; \mathbb{R}^d): S_{T_1 T_2} (\varphi)
	\leq s \} . \]
	\begin{prop}\label{prop:LDP}
		Let $T$ and $s_0$ be positive constants and $F \subset \mathbb{R}^d$ be a fixed compact set. Suppose that following conditions hold:
		\begin{trivlist}
			\item[i).] There exists a function $U \in C^1(\R^d;\R^+)$ and three positive constants $\zeta,
			\kappa, M $ such that following two inequalities hold:
			\begin{equation}\label{ash2}
				\nabla U (x) \cdot b (x) \leq - \zeta  | \nabla U (x) |^2, \quad 
				x \in \mathbb{R}^d,
			\end{equation}
			and
			\begin{equation}\label{ash3}
				\nabla U(x) \cdot \frac{x}{|x|} \geq \kappa,\quad
				\forall x \in B_0^c(M).
			\end{equation}
			\item[ii).] 
			$\sigma$ is bounded by $\bar{\lambda}>0$ and there exists a positive constant $ \underline{\lambda}$ such that the eigenvalues of $\sigma\sigma^T$ is larger than $ \underline{\lambda}$. 
			\item[iii).] Let $\tilde{U}  \in C^2 (\mathbb{R}^d;\mathbb{R}^+)$ be a function satisfying $\lim_{| x | \uparrow + \infty} \tilde{U} (x) = + \infty $. There exist two constants $\varepsilon_1, \chi \in (0, +
			\infty)$ such that for any $\varepsilon \in (0, \varepsilon_1)$ and $x \in B_0^c(M) $, we have
			\begin{equation}\label{ash4}
				\frac{\varepsilon^2}{2} \sum_{i, j} a_{i j} (x) \frac{\partial^2}{\partial
					x_i \partial x_j} \tilde{U} (x) + \nabla \tilde{U} (x) \cdot b(x)  < -
				\chi,
			\end{equation}
			where $(a_{i j} )_{d \times d}=\sigma \sigma^T $.
		\end{trivlist}
		Then $\{ P^{\varepsilon}_x \}$ satisfies the large deviations principle on
		$C_x ([0, T] ; \mathbb{R}^d)$ with good rate function $
		S_{0 T} (\cdot)$, uniformly with respect to $x
		\in F$ and $s \in [0, s_0)$ as $\varepsilon \downarrow 0$. Precisely speaking,
		for any fixed compact set $F\subset \R^d$ and any $s_0,\delta,
		\gamma > 0$, there exists an $\varepsilon_0 = \varepsilon_0 (\gamma, \delta,
		T, F, s_0) > 0$ such that following two estimates are true:
		\begin{equation}\label{ldpl}
			P^{\varepsilon}_x (\rho_{0 T} (X_{\cdot}^{\varepsilon}, \varphi) <
			\delta) \geq \exp \left( - \varepsilon^{- 2} (S_{0 T} (\varphi) +
			\gamma) \right),
		\end{equation}
		for any $x\in F$, $s\in [0,s_0)$, $\varphi \in \Phi_x (s)$ and $\varepsilon \in (0, \varepsilon_0)$, and
		
		\begin{equation}\label{ldpu}
			P^{\varepsilon}_x (\rho_{0 T} (X_{\cdot}^{\varepsilon}, \Phi_x (s)) \geq
			\delta) \leq \exp (- \varepsilon^{- 2} (s - \gamma)), 
		\end{equation}
		for any $x\in F$, $s\in [0,s_0)$ and $\varepsilon \in (0, \varepsilon_0)$.
		
	\end{prop}
	
	To prove Proposition \ref{prop:LDP}, we need the
	following Lemma \ref{lem:ash}. Its idea comes from a similar Lemma in {\cite{gfw}}.

	\begin{lem}\label{lem:ash}
		Suppose that  i), ii) and iii) in Proposition \ref{prop:LDP} hold. 
		Then:
		\begin{trivlist}
			\item[a).]
			\eqref{1.2} has a unique strong solution $X^\eps$ for any $\eps\in(0,\eps_1)$. Moreover, for any $T>0$, $h \in H ([0, T] ; \mathbb{R}^d)$ and $x \in \mathbb{R}^d$, the integral equation
			\begin{equation}\label{aum1}
				\varphi_t = x + \int^t_0 b (\varphi_s) d s + \int^t_0 \sigma (\varphi_s)
				\dot{h}_s d s,\qquad t\in[0,T],
			\end{equation}
			has a unique solution $\varphi \in C_x ([0, T] ; \mathbb{R}^d)$.

			\item[b).] For any $\alpha > 0$, the solution map of equation \eqref{aum1}:
			\begin{eqnarray*}
				S_x (\cdot) : H ([0, T] ; \mathbb{R}^d)&\longrightarrow& C_x ([0, T] ; \mathbb{R}^d)\\
				h&\longmapsto& \ph
			\end{eqnarray*}
			is continuous on $K_{\alpha} := \{
			\| h \|_1 \leq \alpha \}$.
			
			\item[c).] (quasi-continuous)
			For any positive constants $\rho, \alpha, c, R $, there exists an $\varepsilon_0>0$ and a $\beta > 0$ such that 
			\begin{equation*}
				P (\rho_{0 T} (X_{\cdot}^{\varepsilon}, \varphi_{\cdot}) > \rho, \quad \rho_{0
					T} (\varepsilon W_{\cdot}, h_{\cdot}) \leq \beta) \leq \exp (-
				\varepsilon^{- 2} R)
			\end{equation*}
			holds for any $\varepsilon \in (0,\varepsilon_0)$ and $h$, $x$ satisfying $\| h \|_1 \leq \alpha$ and $| x | \leq c$.
			Here $\varphi$ is $S_x (h)$.
		\end{trivlist}
	\end{lem}
	
	\begin{proof}
		We prove the Lemma part by part.
		\begin{trivlist}
			\item[\bf a):]
			Let us set $S^{\varepsilon}_n = \inf \{ t : | X^{\varepsilon}_t | \geq
			n \}, \forall \ n \in \mathbb{N}^+, \varepsilon \in (0, \varepsilon_1)$.
			Because of the locally Lipschitz continuity of $b$ and $\sigma$, Theorem 5.2.5 in {\cite{gtm113}} shows that 
			\[ X_{t \wedge S^{\varepsilon}_n}^{\varepsilon} = x_0 + \int^t_0 b (X_{s
				\wedge S^{\varepsilon}_n}^{\varepsilon}) d s + \varepsilon \int^t_0
			\sigma (X_{s \wedge S^{\varepsilon}_n}^{\varepsilon}) d W_s \]
			has a unique strong solution, for any $ n \in \mathbb{N}^+$ and $\varepsilon \in (0, \varepsilon_1)$. Thus, if
			\begin{equation}\label{ash9}
				P \left(\lim_{n \uparrow + \infty} S^{\varepsilon}_n < t\right) = 0, \qquad\forall \ t \in (0, + \infty),
			\end{equation}
			we can prove that \eqref{1.2} has a unique strong solution for any $\varepsilon \in (0,\varepsilon_1)$. According to condition iii) of Proposition \ref{prop:LDP}, we can prove \eqref{ash9} by taking $\tilde{U} (x) + C$ as $V (t, x)$ in Theorem 3.5 of {\cite{km}}, where $C$ is a sufficiently large constant.
			
			Because of the locally Lipschitz continuity of the
			coefficients, by the Picard iteration we know that this integral
			equation \eqref{aum1} has a unique continuous local solution.
			
			Now it is sufficient to show that $\ph$ is bounded on $[0,T]$. We prove it by contradiction. If not,
			for the smallest critical point $t_0 \in [0, T]$, i.e. $t_0=\sup\{t \in [0, T]:|\ph(t)|<\infty\}$, there
			must exist a sequence of $\{ t_n \}$ such that $\{t_n\}$ strictly increases and converges to $t_0$ as $n\rightarrow\infty$, so we have $\lim_{n \rightarrow + \infty} |
			\varphi_{t_n} | = + \infty$. We set $U^{\ast} = \max_{| x | = M} U (x)$,
			$U_{\ast} = \min_{| x | = M} U (x)$, $U^{\ast 0} = \max_{| x | = | \varphi_0
				|} U (x)$ and $U_{\ast 0} = \min_{| x | = | \varphi_0 |} U (x)$. We choose a particular $n\in\bb{N}$, such that $| \varphi_{t_n} | > (M \vee | \varphi_0 |) + \frac{
				{\bar{\lmd}}^2}{2 \zeta \kappa} | h |_1 +\left(  \frac{U^{\ast} - U_{\ast}}{\kappa}
			\vee \frac{U^{\ast 0} - U_{\ast 0}}{\kappa}\right)  + 1$. We set $M_n = \max_{t \in [0,
				t_n]} | \varphi_t | < + \infty$. Since $U\in C^1(\R^d;\R^+)$, $U$ is bounded and
			Lipschitz continuous on $\overline{B_0(M_n)}$. Because $\varphi$ is
			absolutely continuous on $[0, t_n]$, $U (\varphi)$ is
			absolutely continuous on $[0, t_n]$.
			
			We denote
			$\tilde{t} = \max_{t \in [0, t_n]} \{ t : | \varphi_t | \leq M \}$, when $| \varphi_0 | \leq M$.
			From \eqref{ash2} in i) and the boundness of $\sigma$ in ii), we have 
			\begin{eqnarray*}
				U (\varphi_{t_n}) - U (\varphi_{\tilde{t}}) 
				& = & \int^{t_n}_{\tilde{t}}
				\frac{d U (\varphi_s)}{d s} d s\label{eq:Uphtn-Uphtt}\\
				&=&\int^{t_n}_{\tilde{t}}\llangle \nabla U(\ph(s)),  b(\ph(s))+\sigma(\ph(s))\dot h\rrangle_{\R^d} d s   \\
				&\leq&  \int^{t_n}_{\tilde{t}} - \zeta  | \nabla U (\varphi_s) |^2
				+ \bar{\lmd} | \nabla U (\varphi_s) | | \dot{h}_s | d s\nonumber\\
				& \leq & \int^{t_n}_{\tilde{t}} - \frac{\zeta}{2}  | \nabla U
				(\varphi_s) |^2 + \frac{ {\bar{\lmd}}^2}{2 \zeta} | \dot{h}_s |^2 d s\nonumber\\
				&\leq&  \frac{ {\bar{\lmd}}^2}{2 \zeta} \| h \|_1^2,\nonumber
			\end{eqnarray*}
			which implies that
			\begin{equation}\label{ash5}
				U (\varphi_{t_n}) \leq \frac{ {\bar{\lmd}}^2}{2 \zeta} \| h \|_1^2 + U (\varphi_{\tilde{t}}) 
				\leq \frac{ {\bar{\lmd}}^2}{2 \zeta} \| h \|_1^2 + U^{\ast} .
			\end{equation}
			On the other hand, by \eqref{ash3}, for any $y \in \{ x : | x | > M \}$, there exists a $\xi\in \R^d$ such that $\xi$ is a convex combination of $ y $ and $ \frac{M}{| y |} y $, and we have
			\begin{equation}\label{ash6}
				U (y) - U \left( \frac{y}{| y |} M \right) = \nabla U (\xi) \cdot \left( y
				- \frac{y}{| y |} M \right)  \geq \kappa \left| y - \frac{y}{| y |} M
				\right|=\kappa(|y|-M).
			\end{equation}
			Because of \eqref{ash6} and $|\varphi_{t_n} |>M$, we have
			\begin{equation}\label{ash7}
				U (\varphi_{t_n}) \geq \kappa (| \varphi_{t_n} | - M) + U_{\ast}.
			\end{equation}
			Hence, by comparing \eqref{ash5} and \eqref{ash7} we have
			\[ | \varphi_{t_n} | \leq M + \frac{ {\bar{\lmd}}^2}{2 \zeta \kappa} \| h
			\|_1^2 + \frac{U^{\ast} - U_{\ast}}{\kappa}, \]
			which is contrary to the definition of $\varphi_{t_n}$.
			
			For the similar reason as before, if $| \varphi_0 | > M$, then we have
			\[ | \varphi_{t_n} | \leq | \varphi_0 | + \frac{ {\bar{\lmd}}^2}{2 \zeta
				\kappa} \| h \|_1^2 + \frac{U^{\ast 0} - U_{\ast 0}}{\kappa}, \]
			which is also contrary to the definition of $\varphi_{t_n}$.
			
			Consequently,
			we have
			\begin{equation}\label{ash8}
				| \varphi_t | \leq \left( M \vee | \varphi_0 |\right)  + \frac{ {\bar{\lmd}}^2}{2 \zeta
					\kappa} \| h \|_1^2 + \left( \frac{U^{\ast} - U_{\ast}}{\kappa} \vee \frac{U^{\ast
						0} - U_{\ast 0}}{\kappa}\right) , \qquad\forall t \in [0, T], 
			\end{equation}
			which implies that \eqref{aum1} has a unique solution in $C_x ([0, T] ;
			\mathbb{R}^d)$.

			\item[\bf b):]
			According to \eqref{ash8}, for any fixed
			initial point $x$ and $\{ h_n \} \subset K_{\alpha}$, $\{ S_x (h_n)
			\}$ are bounded in $[0, T].$ Thus, $b (S_x (h_n)) \in C ([0, T] ;
			\mathbb{R}^d)$ and $\sigma (S_x (h_n)) \in C ([0, T] ; \mathbb{R}^{d \times
				d})$ are bounded and Lipschitz continuous. Therefore, the proof of Lemma 2.5
			in {\cite{gfw}} can tell us that b) is true.
			
			\item[\bf c):]
			The proof of Theorem 2.9 in {\cite{gfw}}
			shows that c) is true for bounded Lipshitz continuous $b$ and $\sigma$.
			The following proof of general case comes from {\cite{gfw}}. We
			write it in more details.
			
			From \eqref{ash8}, under the conditions $\| h \|_1 \leq \alpha$ and $| x
			|=|\ph_0| \leq c$, we know that there exists a positive constant $C$ such that $|
			\varphi_t | \leq C$ for all $t \in [0, T]$. For any $\rho > 0$, we define
			\[ \tilde{b} (x) = \left\{ \begin{array}{l}
				b (x), \hspace{7em} | x | \leq C + 2 \rho,\\
				b \left( \frac{x}{| x |} (C + 2 \rho) \right), \qquad | x | > C + 2
				\rho,
			\end{array} \right. \]
			and
			\[ \tilde{\sigma} (x) = \left\{ \begin{array}{l}
				\sigma (x), \hspace{7em} | x | \leq C + 2 \rho,\\
				\sigma \left( \frac{x}{| x |} (C + 2 \rho) \right), \qquad | x | > C +
				2 \rho .
			\end{array} \right. \]
			It is easy to see that $\tilde{b}$ and $\tilde{\sigma}$ are bounded
			Lipschitz continuous. We set
			\[ \tilde{X}^{\varepsilon}_t = x + \int^t_0 \tilde{b}
			(\tilde{X}^{\varepsilon}_s) d s + \int^t_0 \varepsilon \tilde{\sigma}
			(\tilde{X}^{\varepsilon}_s) d W_s \]
			and
			\[ \tilde{\varphi}_t = x + \int^t_0 \tilde{b} (\tilde{\varphi}_s) d s +
			\int^t_0 \varepsilon \tilde{\sigma} (\tilde{\varphi}_s) \dot{h}_s d s. \]
			It is clear that $X_t^{\varepsilon}$ is equal to $\tilde{X}^{\varepsilon}_t$
			in indistinguishable sense up to exiting the ball $\{ | x | \leq C + 2 \rho
			\}$. Moreover, we have $\varphi_t =
			\tilde{\varphi}_t$ for $t \in [0, T]$, since $| \varphi_t | < C$ in $[0, T]$. For any $X_{\cdot}^{\varepsilon}
			(\omega) \in \{ \rho_{0 T} (X_{\cdot}^{\varepsilon}, \varphi_{\cdot}) >
			\rho, \rho_{0 T} (\varepsilon W_{\cdot}, h_{\cdot}) \leq \beta \}$, there
			exists a $t_0 \in [0, T]$ satisfying $| X_{t_0}^{\varepsilon} (\omega) -
			\varphi_{t_0} | > \rho$. Thus, there exists a $\eta \in (0, 0.5 \rho)$ such
			that $| X_{t_0}^{\varepsilon} (\omega) - \varphi_{t_0} | > \rho + \eta$.
			Because of~$X_0^{\varepsilon} = \varphi_0 = x$, there must exist a
			$t_1 \in [0, t_0]$ satisfying \ $| X_{t_1}^{\varepsilon} (\omega) -
			\varphi_{t_1} | = \rho + \frac{\eta}{2}$. Meanwhile, we have $X_{t_1}^{\varepsilon}
			(\omega) \in \{ x : | x | \leq C + 2 \rho \}$, which implies
			$\tilde{X}^{\varepsilon}_{t_1} (\omega) = X_{t_1}^{\varepsilon} (\omega)$.
			Hence, we have
			\[ \{ \rho_{0 T} (X_{\cdot}^{\varepsilon}, \varphi_{\cdot}) > \rho, \rho_{0
				T} (\varepsilon W_{\cdot}, h_{\cdot}) \leq \beta \} = \{ \rho_{0 T}
			(\tilde{X}_{\cdot}^{\varepsilon}, \tilde{\varphi}_{\cdot}) > \rho,
			\rho_{0 T} (\varepsilon W_{\cdot}, h_{\cdot}) \leq \beta \} . \]
			By virtue of the property that $\tilde{X}^{\varepsilon}_t$ and
			$\tilde{\varphi}_t$ satisfy c), we finish the proof.
		\end{trivlist}
	\end{proof}

	\begin{rem}\label{rem:ase}
		In the case that $b$ is bounded Lipschitz continuous and
		$\sigma$ is a constant, a),b) and c) are true. In fact, a) in Lemma \ref{lem:ash} is clearly true.
		Moreover, for this case, \eqref{1.2} can be pathwisely regarded as \eqref{aum1} for almost every $\omega\in \Omega$. Therefore, its solution is $X_\cdot^\eps=S_x(\eps W_\cdot)$ almost surely. 
		Furthermore, it can be proved by Gronwall inequality that the solution map $S_x
		(\cdot)$ is actually a uniformly continuous map from $ C_0([0,T];\R^d)$ to $ C_x([0,T];\R^d)$. Thus, b) and c) in Lemma \ref{lem:ash} are also true.
	\end{rem}

	By the result of Lemma \ref{lem:ash}, \cite{gfw} proves that the strong solution of \eqref{1.2} satisfies the LDP with respect to an initial point. However, it is not difficult to get the conclusion that the strong solution of \eqref{1.2} satisfies the LDP with respect to the initial point uniformly in any compact set from \cite{gfw}. 
	Now we can give a proof of  Proposition \ref{prop:LDP}, which is based on arguments in \cite{gfw}.
	
	\begin{proof}[\bf Proof of Proposition \ref{prop:LDP}]
		To show that $S_{0 T} (\varphi)$ is a good rate function, we need to prove
		that for any $\alpha > 0$, $\{ \varphi : S_{0 T} (\varphi) \leq \alpha \}$ is
		a compact set in $C_x ([0, T] ; \mathbb{R}^d)$. Indeed, by condition ii), we know that $\sigma$ is always invertible. Thus, $\{ \varphi
		: S_{0 T} (\varphi) \leq \alpha \} \subset C_x ([0, T] ; \mathbb{R}^d)$ is
		the image of the compact set $\left\{ h : \frac{1}{2} \| h \|^2_1 \leq \alpha
		\right\} \subset C ([0, T] ; \mathbb{R}^d)$ under $S_x (\cdot)$. According
		to condition b) of Lemma \ref{lem:ash}, we know that $S_x (\cdot)$ is a continuous map
		on $\left\{ h : \frac{1}{2} \| h \|^2_1 \leq \alpha \right\}$. Thus, $\{
		\varphi : S_{0 T} (\varphi) \leq \alpha \}$ is a compact set in $C_x ([0, T]
		; \mathbb{R}^d)$.
		
		By the proof of Theorem 2.4 in {\cite{gfw}}, following results are
		obtained: For any $x \in F$, every closed set $C \subset C_x ([0, T] ;
		\mathbb{R}^d)$ and every open set $G \subset C_x ([0, T] ; \mathbb{R}^d)$
		satisfying
		\[ \inf_{\varphi \in C} S_{0 T} (\varphi) < s_0, \inf_{\varphi \in G} S_{0
			T} (\varphi) < s_0, \]
		we have:
		\begin{trivlist}
			\item [i).]
			For any $x \in F$, there exists an $\varepsilon_0 = \varepsilon_0
			(F, s_0, C)$ such that
			\begin{equation}\label{ldpc}
				\lim \sup_{\varepsilon \rightarrow 0} \varepsilon^2 \ln P^{\varepsilon}_x
				(X_{\cdot}^{\varepsilon} \in C) \leq -
				\inf_{\varphi \in C} S_{0 T} (\varphi)
			\end{equation}
			for any $\varepsilon \in (0, \varepsilon_0) .$
			
			\item [ii).]
			For any $x \in F$, there exists an $\varepsilon_0 = \varepsilon_0
			(F, s_0, G)$ such that
			\begin{equation}\label{ldpo}
				\lim \inf_{\varepsilon \rightarrow 0} \varepsilon^2 \ln P^{\varepsilon}_x
				(X_{\cdot}^{\varepsilon} \in G) \geq -
				\inf_{\varphi \in G} S_{0 T} (\varphi)
			\end{equation}
			for any $\varepsilon \in (0, \varepsilon_0) .$
		\end{trivlist}
		
		From the standard fact of the LDP theory, we know that \eqref{ldpc}
		and \eqref{ldpo} imply \eqref{ldpl} and \eqref{ldpu}
		respectively.
		
	\end{proof}

\medskip

\section{Stability in the sense of quasi-potential}

In this section, we introduce the definition of stability in the sense of quasi-potential and its properties. These properties allow us to connect the stability in the sense of quasi-potential with the trajectory property of \eqref{1.1}. 

The definitions and notations about \eqref{1.2} below are from \cite{fw}.

\begin{defn}(\cite{fw})
	By the definition of $S_{T_1 T_2} (\cdot)$, for any $x, y
	\in \mathbb{R}^d$, the quasi-potential with respect to $x$ is defined as
	\[ V (x, y) = \inf \{ S_{T_1 T_2} (\varphi) : \varphi \in C ([T_1, T_2] ;
	\mathbb{R}^d), \varphi_{T_1} = x, \varphi_{T_2} = y, 0 \leq T_1 \leq T_2 \} . \]
	In the same way, for a set $D \subseteq \mathbb{R}^d$ and any $x, y
	\in D$, we define 
	\[ V_D (x, y) = \inf \{ S_{T_1 T_2} (\varphi) : \varphi \in C ([T_1, T_2] ;
	\bar{D}), \varphi_{T_1} = x, \varphi_{T_2} = y, 0 \leq T_1 \leq T_2 \} . \]
	For $x, y \in D$, the notation $x \thicksim_D y$ means $V_D (x, y) = V_D (y,
	x) = 0$. Obviously, this is an equivalent relation in $D$. A set $C \subseteq
	D$ is called an equivalent set, if for any $x, y \in C$, we have $x \thicksim_D
	y$. For the sake of simplicity, we replace $x
	\thicksim_D y$ by $x \thicksim  y$, if $D =\mathbb{R}^d$.
	
	Let $K_1, K_2, \ldots , K_l$ be different equivalent sets in $\mathbb{R}^d$, which
	satisfy $K_i \cap K_j = \varnothing$ and $x \nsim y$, for any $x \in K_i, y \in
	K_j, i \neq j$. We define
	\[ V (K_i, K_j) = \inf \{ S_{T_1 T_2} (\varphi) : \varphi \in C ([T_1, T_2] ;
	\mathbb{R}^d), \varphi_{T_1} \in K_i, \varphi_{T_2} \in K_j, 0 \leq T_1 \leq T_2
	\} \]
	and
	\[ V_D (K_i, K_j) = \inf \{ S_{T_1 T_2} (\varphi) : \varphi \in C ([T_1, T_2]
	; \bar{D}), \varphi_{T_1} \in K_i, \varphi_{T_2} \in K_j, 0 \leq T_1 \leq T_2 \} .
	\]
	Furthermore, we define
	\[ \tilde{V}  (K_i, K_j) = \inf \left\{ S_{T_1 T_2} (\varphi) : \varphi \in C
	\left( [T_1, T_2] ; \mathbb{R}^d \backslash \bigcup_{s \neq i, j} K_s
	\right), \varphi_{T_1} \in K_i, \varphi_{T_2} \in K_j, 0 \leq T_1 \leq T_2
	\right\} . \]
	If there is no such $\varphi$, then we set $\tilde{V}  (K_i, K_j) = + \infty$.
	We also define
	\[ \tilde{V}_D (K_i, K_j) = \inf \left\{ S_{T_1 T_2} (\varphi) : \varphi \in C
	\left( [T_1, T_2] ; \bar{D} \backslash \bigcup_{s \neq i, j} K_s \right),
	\varphi_{T_1} \in K_i, \varphi_{T_2} \in K_j, 0 \leq T_1 \leq T_2\right\} . \]
	If there is no such $\varphi$, then we set $\tilde{V}_D (K_i, K_j) = +\infty$.
	
	Let $C \subset \mathbb{R}^d$ be an arbitrary set. We use $\partial C$ to represent the boundary of $C$.
	For any $\delta > 0$, we set 
	
	$$C_{\delta} = \{ x :\exists y \in C,\ | x - y | < \delta \}$$
	and
	\[ C_{- \delta} = \{ x : x \in C,\ \min_{y \in \partial C}  | x - y | >\delta \} . \]
	Furthermore, if $C$ is a compact set with smooth boundary, for a point $x$ lying between
	$\partial C$ and $\partial (C_{- \delta})$, we denote the closest point on
	$\partial (C_{- \delta})$ by $x_{- \delta}$.
\end{defn}

\begin{defn}(\cite{fw})
	A set $C \subset \mathbb{R}^d$ is called a {\tmstrong{stable set}} of equation
	\eqref{1.2}, if $V (x, y) > 0$ for any $x \in C$ and any $y \nin C$. Otherwise,
	$C$ is called an {\tmstrong{unstable set}}.
\end{defn}

In the proof of the main result, we will use some important assumptions. For the sake of clarity, we list them here. Let us set $\mathcal{L}= \{ 1, 2, \ldots, l \}$ and introduce the following assumption.

\begin{assumption}\label{mainassumption}
	$\,$
	\begin{trivlist}
		\item[1).] There is a finite number of compact
		equivalent sets $K_1, K_2, \ldots, K_l$, all of which are contained in $B_0(M)$, satisfying $x \nsim y$ for $x \in K_i,\, y \in K^c_i, \, i \in
		\mathcal{L}$. Furthermore, every $\omega$-limit set of \eqref{1.1} is
		contained entirely in one of the $K_i,\; i \in \mathcal{L}$.
		
		\item[2).] Let $U$ be the one in  i) of Proposition \ref{prop:LDP}.
		For any $x \in \mathbb{R}^d \backslash \cup_{i \in \mathcal{L}}
		K_i$,
		we suppose $\nabla U (x) \neq 0$. Moreover, 
		there exists a constant $\tilde{\delta} > 0$ and two constants
		$k_1 \ge k_2>0 $, such that for any $\delta', \delta''$ satisfying $0 <
		(\frac{k_1}{k_2})^{\frac{1}{2}} \delta'' < \delta' < \tilde{\delta}$, we have
		\begin{equation}\label{con23}
			\min_{x \in \partial (K_i)_{\delta''}, y \in \partial (K_i)_{\delta'}} \Big(U
			(y) - U (x)\Big) \geq k_2 (\delta')^2 - k_1 (\delta'')^2,
		\end{equation}
		for any stable $K_i, i \in \mathcal{L}$.

		\item[3).] Let $\tilde{U} (x) \in C^2 (\mathbb{R}^d;\mathbb{R}^+)$ be a function satisfying $\lim_{| x | \uparrow + \infty} \tilde{U} (x) = + \infty $. We suppose that for any small $\delta>0$, there exist two constants 
		$\varepsilon_{\delta}, \chi_{\delta} \in (0, + \infty)$, such that we have 
		\begin{equation}\label{con25}
			\frac{\varepsilon^2}{2} \sum_{i, j} a_{i j} (x) \frac{\partial^2}{\partial
				x_i \partial x_j} \tilde{U} (x) + \nabla \tilde{U} (x) \cdot b(x) < -
			\chi_{\delta},
		\end{equation}
		for any $\varepsilon \in (0, \varepsilon_{\delta})$ and $x \in \mathbb{R}^d \backslash \cup_{i \in \mathcal{L}} \left(K_i\right)_{\delta}$.
		Here $(a_{i j} (x))_{d \times d}$ is $\sigma \sigma^T (x)$.
	\end{trivlist}	
\end{assumption}

\begin{rem}\label{rem:generalpositivecon}
	Let  $H(U(x))$  denote the Hessian matrix of $U$ at $x$. If $K_i$ is a stable point, then the positive definiteness of $H(U(K_i))$ guarantee that \eqref{con23} holds. In general, for any $n\in\mathbb{N}^{+}$, \eqref{con23} can be replaced by 
	\begin{equation*}
		\min_{x \in \partial (K_i)_{\delta''}, y \in \partial (K_i)_{\delta'}} \Big(U
		(y) - U (x)\Big) \geq k_2(\delta')^{2n} - k_1 (\delta'')^{2n},
	\end{equation*}
	for any  $0 < (\frac{k_1}{k_2})^{\frac{1}{2n}}  \delta'' < \delta' < \tilde{\delta}$.
\end{rem}

\begin{rem}\label{rem:con2simplification}
	If $U$ satisfies the following condition:
	\begin{trivlist}
		\item[3)$'$.]Let $\chi$ and $\varepsilon_1$ be two positive constants. We suppose that $U\in C^2 (\mathbb{R}^d;\mathbb{R}^+)$ satisfies 
		\begin{equation}
			\frac{\varepsilon^2}{2} \sum_{i, j} a_{i j} (x) \frac{\partial^2}{\partial
				x_i \partial x_j} U(x) -\zeta \left|\nabla U(x) \right|^2 < -
			\chi,
		\end{equation}
		for any $\varepsilon\in \left( 0, \varepsilon_1\right) $ and $ x \in B_0^c(M)$. 
	\end{trivlist}
	Then we can take $U(x)$ as $\tilde{U}(x)$ and reduce 3) of Assumption \ref{mainassumption} to 3)$'$.
\end{rem}	

\begin{lem}\label{lem:lowerboundedofqp}
	Under conditions i) and ii) of Proposition \ref{prop:LDP}, the quasi-potential of equation \eqref{1.2} has a lower estimate
	\begin{equation}\label{lowerboundedofqp}
		V (x, y) \geq \frac{2 \zeta}{{\bar{\lmd}}^2} (U (y) - U (x)),
	\end{equation}
	for any $x, y \in \mathbb{R}^d$.
\end{lem}

\begin{proof}
	By i) and ii) of Proposition \ref{prop:LDP}, for any $x, y \in \mathbb{R}^d$, $T \in [0, + \infty)$ and $\varphi \in \{\varphi\in AC_x ([0, T] ; \mathbb{R}^d):\varphi_T=y\}$, we have
	\begin{eqnarray*}
		S_{0 T} (\varphi) & = & \frac{1}{2} \int^T_0 | \sigma^{- 1} (\varphi_t)
		(\dot{\varphi}_t - b (\varphi_t)) |^2 d t\\
		& \geq & \frac{1}{2 {\bar{\lmd}}^2} \int^T_0 | \dot{\varphi}_t - b
		(\varphi_t) |^2 d t\\
		& = & \frac{1}{2 {\bar{\lmd}}^2} \int^T_0 | \dot{\varphi}_t - b
		(\varphi_t) - 2 \zeta \nabla U (\varphi_t) + 2 \zeta \nabla U (\varphi_t)
		|^2 d t\\
		& \geq & \frac{1}{2 {\bar{\lmd}}^2} \int_0^T | \dot{\varphi}_t - b
		(\varphi_t) - 2 \zeta \nabla U (\varphi_t) |^2 d t + \frac{2
			\zeta}{{\bar{\lmd}}^2} \int_0^T (\dot{\varphi}_t, \nabla U (\varphi_t)) d
		t\\
		& \geq & \frac{2 \zeta}{{\bar{\lmd}}^2} (U (\varphi_T) - U (\varphi_0)) =
		\frac{2 \zeta}{{\bar{\lmd}}^2} (U (y) - U (x)) .
	\end{eqnarray*}
	Thus, $V (x, y) \geq \frac{2 \zeta}{{\bar{\lmd}}^2} (U (y) - U (x))$.
	
\end{proof}

\begin{cor}\label{cor:infinitycost}
	Let us suppose that the conditions i) and ii) of Proposition \ref{prop:LDP} hold. Then we have
	\begin{equation}\label{eq:infinitycost}
		\lim_{| y | \uparrow + \infty} V (x, y) = + \infty, \quad x\in \R^d.
	\end{equation}
\end{cor}

The following Corollary \ref{rem:howtogetstable} and Proposition \ref{prop:sstoos} provide a way to determine whether a set is stable or not. This method will be used in Example \ref{Bernoulli}, Example \ref{Non-symmetrical Example} and Example \ref{duffing}.

\begin{cor}\label{rem:howtogetstable}
	Suppose that conditions i) and ii) of Proposition \ref{prop:LDP} hold. 
	If for some $i\in\cl L$, 
	there exists  $\delta>0$, such that for any $x \in K_i, y \in (K_i)_{\delta} \backslash K_i $, we have $U(y)-U(x) > 0$, then $K_i$ is a stable set.
\end{cor}

The following lemma from {\cite{fw}} illustrates the continuity of the
quasi-potential.

\begin{lem}\label{lem:qpcontinuous}(\cite{fw})
	Let $C \subset \mathbb{R}^d$ be a convex compact set. Then there exists a constant $L=L(C)$, such that for any
	$x, y \in C$,  there exists  $\varphi \in C ([0, T] ; \mathbb{R}^d)$
	with $\varphi_0 = x, \varphi_T = y, T = | x - y |$
	and
	\[ S_{0 T} (\varphi) \leq L | x - y | . \]
\end{lem}

\begin{proof}
	We choose $\varphi$ as
	\[ \varphi_t = x + \frac{t}{| x - y |} (y - x), \quad t \in [0, | x - y |], \]
	then by the local boundness of coefficients $b$ and $\sigma$, we can finish the proof.
	
\end{proof}

\begin{prop}\label{prop:sstoos}
	If $K \subset \mathbb{R}^d$ is a stable set of equation \eqref{1.2},
	then for any $\delta > 0$, there exists a $\delta' \in (0, \delta)$ such that
	the solution of equation \eqref{1.1} starting at $x_0 \in K_{\delta'}$ does not
	leave $K_{\delta}$.
\end{prop}

\begin{proof}
	We prove this by contradiction. If not, there exists
	a $\delta > 0$ such that for every $n \in \mathbb{N}^+$ satisfying
	$\frac{1}{n} < \delta$ there exists an $x_n \in K_{\frac{1}{n}}$ and a $t_n$
	satisfying $X_{t_n} (x_n) \in K_{\delta + 1} \backslash K_{\delta}$. Thus,
	there exists an $x^{\ast} \in K$ and a $y^{\ast} \in K_{\delta + 1}
	\backslash K_{\delta}$ such that we can choose a subsequence of $\{ x_n \}$
	still denoted by $\{ x_n \}$, satisfying $x_n \rightarrow x^{\ast}$ and
	$X_{t_n} (x_n) \rightarrow y^{\ast}$ as $n \uparrow +\infty$.
	
	We claim that $V (x^{\ast}, y^{\ast}) = 0$. Because of Lemma \ref{lem:qpcontinuous},
	for any $\epsilon > 0$ there exists an $n$ such that $V (x^{\ast}, x_n) <
	\epsilon$ and $V (X_{t_n} (x_n), y^{\ast}) < \epsilon$. Therefore, we have
	$V (x^{\ast}, y^{\ast}) = V (x^{\ast}, x_n) + V (x_n, X_{t_n} (x_n)) + V
	(X_{t_n} (x_n), y^{\ast}) < 2 \epsilon$. Hence, we have $V (x^{\ast},
	y^{\ast}) = 0$, which contradicts to the fact that $K$ is a stable set.
	
\end{proof}

\begin{rem}\label{rem:ssandos}
	Proposition \ref{prop:sstoos} shows that a stable set defined
	by the quasi-potential is also a stable set in the sense of deterministic
	dynamical systems.
\end{rem}

\begin{prop}\label{prop:existoness}
	Under condition 1) of Assumption \ref{mainassumption}, there exists at least one stable $K_i$ for some $i \in \cl L$.
\end{prop}

\begin{proof}
	We prove this result by contradiction. If all the $K_i$ are unstable, then
	there exists an $x_i \in K_i$ and a $y_i \in K^c_i$
	satisfying $V (x_i, y_i) = 0$ for any $i\in\cl L$. Since all the $\omega$-limit sets are in
	$\cup_{s \in \mathcal{L}} K_s$, there must be a set $K_j$ satisfying $V (y_i,
	K_j) = 0$. 
	
	If $j=i$, then we have $y_i \sim K_i$, which is
	contradict to condition 1) of Assumption \ref{mainassumption}. If $j\neq i$, then
	we have $V (K_i, K_j) = 0$. Repeating the above steps, 
	there is a pair  $(K_m, K_n)$ admitting $V (K_m, K_n) = V
	(K_n, K_m) = 0$ because of the finiteness of the set $\cl L$. Condition 1) of Assumption \ref{mainassumption} implies $l = 1$. Therefore, we only
	need to exclude the case that there is only one unstable $K_1$.
	
	If this is the case, there exists an $x \in K_1$ and a $y \in \mathbb{R}^d
	\backslash K_1$ such that $V (x, y) = 0$. However, by 1) of Assumption \ref{mainassumption}, $K_1$ contains all $\omega$-limit sets
	of equation \eqref{1.1}. Thus, $X_t (y)$ converges into $K_1$ as $t\rightarrow\infty$, which implies $V
	(y, K_1) = 0$. Hence, we have $y \thicksim K_1$, which is contradict to condition
	1) of Assumption \ref{mainassumption}.
	
\end{proof}

\medskip

\section{The concentration of the weak limitation of $\{ \mu_{\varepsilon} \}$}

In this section, we show concentration phenomena of $\{\mu_{\varepsilon} \}$ by the LDP method. To construct invariant measures $\{\mu_\eps\}$, we use the method in \cite{km} by the aid of Markov chains. Transition probabilities of Markov chains are estimated in Proposition \ref{prop:transpb}. Lemma \ref{lem:qponcompactset}-Lemma \ref{lem:fw2} are preparations for Proposition \ref{prop:transpb}.

From now on, we assume $\{\mu_{\varepsilon}\}$ has a weak limit or its subsequence has a weak limit $\mu$. This property need the tightness of $\{\mu_{\varepsilon}\}$. Corresponding results can be found in \cite{chen}.

\begin{lem}\label{lem:qponcompactset}
	Suppose conditions i) - ii) of Proposition \ref{prop:LDP} are true. Then there exists a bounded domain $D \supset \cup_{i \in \mathcal{L}} K_i$ with smooth boundary such that for any $i, j \in \{ 1, 2, \ldots, l \}$, we have
	\begin{equation}\label{lem:qponcompactset1}
		\tilde{V}_D (K_i, K_j) = \tilde{V} (K_i, K_j).
	\end{equation}
	
\end{lem}

\begin{proof}
	Let us fix $i, j\in \cl L$ arbitrarily.
	\begin{trivlist}
		\item[\bf If $\tilde{V} (K_i, K_j) = + \infty$:] then for every domain $D$ we
		have $\tilde{V}_D (K_i, K_j) \geq \tilde{V} (K_i, K_j) = + \infty$. Thus,
		$\tilde{V}_D (K_i, K_j) = \tilde{V} (K_i, K_j) = + \infty$.
		
		\item[\bf If $\tilde{V} (K_i, K_j) < + \infty$:]  we claim that there exists a
		compact set $D_{i j}$ such that $\tilde{V}_{D_{i j}} (K_i, K_j) =
		\tilde{V} (K_i, K_j)$. By Corollary \ref{cor:infinitycost},   
		there exists an $N^{}_{i j} > M$ such that for any $x \in
		\{ x : | x | = M \}, y \in \{ y : | y | > N_{i j} \}$, we have $V (x, y) >
		\tilde{V} (K_i, K_j) + 1$, since i) and ii) in Proposition \ref{prop:LDP} are true.
		For any $\varphi \in \{ \varphi : \varphi \in
		C_x (\mathbb{R}^+ ; \mathbb{R}^d \backslash \cup_{s \neq i, j} K_s), x \in
		K_i, \exists \ \tilde{t} \triangleq \tilde{t} (\varphi) > 0 \text{ s.t. }  |
		\varphi_{\tilde{t}} | > N_{i j} \}$, we set $\epsilon = \sup \{ s : s <
		\tilde{t}, | \varphi_s | = M \}$, then $S_{0 \tilde{t}} (\varphi) \geq
		S_{0 \epsilon} (\varphi) + S_{\epsilon \tilde{t}} (\varphi) > \tilde{V}
		(K_i, K_j) + 1$. On the other hand, by the definition of
		$\tilde{V} (K_i, K_j)$, we know that there exists a sequence of absolutely
		continuous functions $\{ \varphi^{(n)} \}$ and a sequence of positive constants $\{ T_n
		\}$, $n \in \mathbb{N}^+$ such that $\varphi_0^{(n)} \in K_i,
		\varphi_t^{(n)} \in \mathbb{R}^d \backslash \cup_{s \neq i, j} K_s, t \in
		[0, T_n]$ and $S_{0 T_n} (\varphi^{(n)}) \leq \tilde{V} (K_i, K_j) +
		\frac{1}{n}$. Let $D_{i j} = \{ x : | x | \leq N_{i j} \}$.
		Thus, for any $n \in \mathbb{N}^+$, we have $\varphi_t^{(n)} \in D_{i j}
		\backslash \cup_{s \neq i, j} K_s, t \in [0, T_n]$, which means
		$\tilde{V}_{D_{i j}} (K_i, K_j) = \tilde{V} (K_i, K_j)$.
		Therefore, we can choose a domain $D \supset \cup_{i, j} D_{i j}$ with
		smooth boundary satisfying \eqref{lem:qponcompactset1}.
	\end{trivlist}
	Hence, we finish the proof.
	
\end{proof}

\begin{lem}\label{lem:incompaceset}
	Suppose conditions i) - ii) of Proposition \ref{prop:LDP} are true.
	Let $C \subset \mathbb{R}^d$ be a compact set and $T, \Theta$ be
	positive constants.  Then there exists a compact set $\Lambda \subset
	\mathbb{R}^d$ such that for any $\varphi \in \{ \varphi \in C_x ([0, T] ;
	\mathbb{R}^d): S_{0 T} (\varphi) \leq \Theta, x \in C \}$, we have
	$\varphi([0,T]) \subset \Lambda$.
\end{lem}

\begin{proof}
	Let $| C | = \sup \{ | x | : x \in C \}$. By Corollary \ref{cor:infinitycost}, for any $x \in \partial B_0(M \vee | C |) $, there exists a constant $\tilde{M}$ sufficiently large such that
	for any $y$ with $| y | > \tilde{M}$, we have
	\[ V (x, y) \geq \frac{2 \zeta}{{\bar{\lmd}}^2} (U (y) - U (x)) > \Theta +
	1. \]
	For any $\varphi \in \{ \varphi \in C_x ([0, T] ; \mathbb{R}^d): x
	\in C, \exists \ \tilde{t} \in [0, T]\ \text{ s.t. }\ | \varphi_{\tilde{t}} | > \tilde{M} \}$, we set $\epsilon (\varphi) = \inf \{ s : s < T, |
	\varphi_s |=M \vee | C | \}$, then
	\[ S_{0 T} (\varphi) \geq S_{0 \epsilon} (\varphi) + S_{\epsilon \tilde{t}}
	(\varphi) \geq \Theta + 1. \]
	Thus, we have finished the proof by $\Lambda =\overline{B_0(\tilde M)}$.
	
\end{proof}

\begin{lem}\label{lem:uppertimeestimate}
	Suppose conditions i) - iii) of Proposition \ref{prop:LDP} and 1)
	of Assumption \ref{mainassumption} are true.
	Let $O \supset \cup_{i \in \mathcal{L}}K_i$ be a bounded open
	set and $C \subset \mathbb{R}^d \backslash O$ be a compact set. Then for any
	$\alpha > 0$, there exist two constants $T, \varepsilon_0 > 0$ such
	that we have
	\[ P^{\varepsilon}_x (\tau^{\varepsilon}_O > T) \leq \exp (- \varepsilon^{-
		2} \alpha) ,\]
	for any $\varepsilon \in (0, \varepsilon_0)$ and $x \in C$.
\end{lem}

\begin{proof}
	Let $\alpha>0$ be fixed arbitrarily.
	
	By Corollary \ref{cor:infinitycost},
	there exists a compact set $C_1 \supset (O \cup C)$ such that for any
	$x \in C$ and $y \in C_1^c$, we have $V (x, y) > \alpha + 1$.

	By 1) of Assumption \ref{mainassumption}, we have 
	$$\tau_{O,x} < + \infty, \quad \forall x\in (C_1\setminus O).$$
	It is easy to
	prove that $\{x\in (C_1\setminus O):\tau_{O,x} \geq a \}$ is a close set for any $a > 0$, since the solution of 
	\eqref{1.1} is continuous with respect to the initial point. Therefore,
	$\tau_{O,x}$ is upper semi-continuous. Thus, $\tau_{O,x}$ can get
	the finite maximum denoted by $T_0$ in $C_1 \backslash O$. 
	
	Let us set $T_1
	= T_0 + 1$. We claim that for any $\varphi \in \{ \varphi \in C_x ([0, T_1] ;
	\mathbb{R}^d \backslash O): x \in C_1 \backslash O \}$, there exists a constant $\theta >
	0$ such that $S_{0 T_1} (\varphi) > \theta$. Otherwise, there exists a
	sequence of $\{ \varphi_n \} \subset \{ \varphi \in C_x ([0, T_1]
	; \mathbb{R}^d \backslash O): x \in C_1\setminus O \}$ such that $S_{0 T_1} (\varphi) <
	\frac{1}{n}$, $n \in \mathbb{N}^+$. By Lemma \ref{lem:incompaceset}, 
	there exists a compact set $\Lambda \subset \mathbb{R}^d \backslash O$ such
	that $ \varphi_n  \subset \Lambda$ for all $n$. Thus, there exists a subsequence of
	$\{ \varphi_n \}$, which converges to some $\varphi^{\ast} \in \{\varphi
	\in C_x ([0, T_1] ; \mathbb{R}^d \backslash O): x \in C_1\setminus O \}$. Therefore, we have $S_{0 T_1} (\varphi^{\ast})
	= 0$, since $S_{0 T_1} (\cdot)$ is lower semi-continuous. Hence, we have $\varphi_t^{\ast} = X_t (\varphi_0^{\ast}) \in  \Lambda
	\subset \mathbb{R}^d \backslash O, \forall t \in [0, T_1]$, which is
	contradict to the definition of $T_0$.
	
	Therefore, for a continuous trajectory $\varphi$ that starts from $C$ and spends time $T$ more than $ 
	T_1$ in $\mathbb{R}^d \backslash O$, we have $S_{0 T} (\varphi) >\theta$.
	In general, for a continuous $\varphi$ that starts from $C$ and spends time $T $ more than $
	T_1$ in $C_1 \backslash O$, we have
	\[ S_{0 T} (\varphi) > \left[ \frac{T}{T_1} \right] \theta > \left(
	\frac{T}{T_1} - 1 \right) \theta, \]
	and $S_{0 T} (\varphi) > \alpha
	+ 1$, if $\varphi$ reaches $C_1^c$.
	
	Thus, there exists a $T_2 > 0$ such that for any $\varphi$
	spending time $T_2$ in $\mathbb{R}^d \backslash O$, we have $S_{0 T_2}
	(\varphi) > \alpha + 1$ anyhow.
	
	By \eqref{ldpc}, because of the closeness of $\{ \varphi \in C_x ([0, T_2] ;
	\mathbb{R}^d \backslash O) \}$ in $C_x ([0, T_2] ;
	\mathbb{R}^d)$, there exists a $\varepsilon_0 > 0$ such
	that we have
	\[ P^{\varepsilon}_x (\tau^{\varepsilon}_O > T_2) \leq \exp (-
	\varepsilon^{- 2} \alpha), \]
	for any $x \in C$ and $\varepsilon \in (0, \varepsilon_0)$.
	
\end{proof}

Let $K_i, i \in \mathcal{L}$ be equivalent sets as in 1) of Assumption \ref{mainassumption}. We set $\delta_1 = \frac{1}{8} \min_{i, j} \tmop{dist} (K_i, K_j)$. For any
$\rho_0 \in (0, \delta_1)$, $\rho_1 \in (0, \rho_0)$ and $\rho_2 \in (0,
\rho_1)$, let us choose open sets $g_i, G_i$ with smooth boundaries
satisfying
\[ \begin{array}{l}
	K_i \subset g_{_i} \subset  (K_i)_{\rho_2} \subset G_i
	\subset (K_i)_{\rho_1}.
\end{array} \]
We denote
\[ g = \bigcup_{i = 1}^l g_i, \quad G = \bigcup_{i = 1}^l
G_i . \]
For any $\varepsilon > 0$, we consider the following two sequences of
stopping times related to $X^{\varepsilon}_t$:
\[ \tau^{\varepsilon}_0 = 0,\quad \sigma^{\varepsilon}_n = \inf \{ t : t \geq
\tau^{\varepsilon}_{n - 1}, X^{\varepsilon}_t \in \partial G \},\quad
\tau^{\varepsilon}_n = \inf \{ t : t \geq \sigma^{\varepsilon}_n,
X^{\varepsilon}_t \in \partial g \}, \quad n \in \mathbb{N}^+ . \]
In order to define the Markov chain
$Z^{\varepsilon}_n = X^{\varepsilon}_{\tau^{\varepsilon}_n}$
appropriately, we need the following lemma. Here, $ \{Z^{\varepsilon}_n\}$ are used to construct $\{\mu_\varepsilon\}$.

\begin{lem}\label{lem:recurrent}
	Under conditions i) - iii) of Proposition \ref{prop:LDP}, we have $P^{\varepsilon}_x (\tau^{\varepsilon}_n <
	+ \infty) = 1$, for any $\varepsilon \in (0, \varepsilon_1)$, $x \in \mathbb{R}^d$ and ${n \in \mathbb{N}^+}$. 
\end{lem}

\begin{proof}
	We prove this lemma by two statements: a). $\forall \ x \in
	\mathbb{R}^d, P^{\varepsilon}_x (\sigma^{\varepsilon}_1 < + \infty) = 1$ and
	b). $\forall \ x \in \partial G, P^{\varepsilon}_x (\tau^{\varepsilon}_1 < +
	\infty) = 1$ and induction on n.
	
	For the case $n=1$, by a), b) and the strong
	Markov property of $X^{\varepsilon}_t$, we have 
	\begin{eqnarray}\label{recurrent1}
		P^{\varepsilon}_x (\tau^{\varepsilon}_1 < + \infty) & = &
		E^{\varepsilon}_x \left(1_{\{ \sigma^{\varepsilon}_1 < + \infty \}} 1_{\{
			\tau^{\varepsilon}_1 - \sigma^{\varepsilon}_1 < + \infty \}}\right) \nonumber\\
		& = & E^{\varepsilon}_x \left(E^{\varepsilon}_x \left(1_{\{ \sigma^{\varepsilon}_1
			< + \infty \}} 1_{\{ \tau^{\varepsilon}_1 - \sigma^{\varepsilon}_1 < +
			\infty \}} | \mathcal{F}_{\sigma^{\varepsilon}_1} \nobracket\right)\right) \nonumber\\
		& = & E^{\varepsilon}_x \left( 1_{\{ \sigma^{\varepsilon}_1 < + \infty
			\}} P_{X_{\sigma^{\varepsilon}_1}^{\varepsilon}}^{\varepsilon}
		(\tau^{\varepsilon}_1 < + \infty) \right) \nonumber\\
		& = & 1,
	\end{eqnarray}
	for any $x \in \mathbb{R}^d$.
	
	For $n \geq 2$, suppose that $P^{\varepsilon}_x (\tau^{\varepsilon}_{n - 1} < + \infty) =
	1$, for any $x \in \mathbb{R}^d$. By the strong Markov property, we have
	\begin{eqnarray*}
		P^{\varepsilon}_x (\tau^{\varepsilon}_n < + \infty) & = &
		E^{\varepsilon}_x (1_{\{ \tau^{\varepsilon}_{n - 1} < + \infty \}} 1_{\{
			\tau^{\varepsilon}_n - \tau^{\varepsilon}_{n - 1} < + \infty \}})
		\nonumber\\
		& = & E^{\varepsilon}_x \left( 1_{\{ \tau^{\varepsilon}_{n - 1} < +
			\infty \}} P_{X_{\tau^{\varepsilon}_{n - 1}}^{\varepsilon}}^{\varepsilon}
		(\tau^{\varepsilon}_1 < + \infty) \right).  
	\end{eqnarray*}
	By the induction
	hypothesis and \eqref{recurrent1}, we have
	\[ P^{\varepsilon}_x (\tau^{\varepsilon}_n < + \infty) = P^{\varepsilon}_x
	(\tau^{\varepsilon}_{n - 1} < + \infty) = 1, \]
	since $X_{\tau^{\varepsilon}_{n - 1}}^{\varepsilon} \in \partial g, a.s. \
	P^{\varepsilon}_x$ is well-defined.
	
	Now we only need to prove a) and b). According to Theorem 3.9 in
	{\cite{km}} and iii) of Proposition \ref{prop:LDP}, by taking $V (s, x)$ in {\cite{km}}
	as $\tilde{U} (x)$ and $U_1$ in {\cite{km}} as $\{ x : | x | < M \}$, we
	know that the process $X^{\varepsilon}_t$ is recurrent with respect to the
	bounded open set $F \triangleq \{ x : | x | < M \}$, for any $\varepsilon \in
	(0, \varepsilon_1)$. Namely, we have $P^{\varepsilon}_x (\tau_{x,F}^{\varepsilon} < +
	\infty) = 1$, for any $x \in B_0^c(M)$ and $\varepsilon
	\in (0, \varepsilon_1)$. Furthermore, Lemma 4.1 in {\cite{km}} tells us
	a) and b) are true. Hence, we finish the proof of this lemma.
	
\end{proof}

The following two lemmata about compact sets from {\cite{fw}} will be used in
the proof of Proposition \ref{prop:transpb}. 
\begin{lem}\label{lem:fw1}
	(\cite{fw}) Let $C$ be a compact set with a smooth boundary.
	For any $\Theta, \gamma > 0$, $\exists \ \bar{\delta} > 0$ such that for any
	$\delta \in (0, \bar{\delta}]$ and any $\varphi \in C ([0, T] ; C)$
	satisfying $T + S_{0 T} (\varphi) \leq \Theta$, there exists a $\bar{\varphi} \in
	C ([0, T] ; C_{- \delta})$ satisfying
	\[ \left\{ \begin{array}{l}
		\bar{\varphi}_0 = \varphi_0, \bar{\varphi}_T = \varphi_T,
		\hspace{3.5em} \tmop{for}\ \varphi_0, \varphi_T \in C_{- \delta}\\
		\bar{\varphi}_0 = (\varphi_0)_{- \delta}, \bar{\varphi}_T =
		(\varphi_T)_{- \delta}, \tmop{otherwise},
	\end{array} \right. \]
	and $S_{0 T} (\bar{\varphi}) \leq S_{0 T} (\varphi) + \gamma$.
\end{lem}

\begin{lem}\label{lem:fw2}
	(\cite{fw}) Let C be a compact subset in $\mathbb{R}^d$.
	Let $K \subset C$ be the max equivalent set which contains $K$. Then for any
	$\delta, \gamma > 0$ and $x, y \in K$, there exists a $ T \in (0, + \infty)$ and a curve $\varphi \in C
	([0, T] ; K_{\delta})$ satisfying $\varphi_0 = x$,
	$\varphi_T = y$ and $S_{0 T} (\varphi) \leq \gamma$.
\end{lem}

The following lemma generalizes Lemma 6.2.1 in {\cite{fw}} from compact
sets to $\mathbb{R}^d$, which is a key lemma in this paper.

\begin{prop}\label{prop:transpb}
	Under conditions i) - iii) of Proposition \ref{prop:LDP} and  1)
	of Assumption \ref{mainassumption}, for any $\gamma > 0$, there exist constants $\rho_2 \in (0,
	\delta_1)$ and  $\varepsilon_0 >0$ such that one
	step transition probabilities of $Z^{\varepsilon}_n$ satisfy inequalities:
	\begin{equation}\label{transeq}
		\exp (- \varepsilon^{- 2} (\tilde{V}  (K_i, K_j) + \gamma)) \leq
		P^{\varepsilon}_x (Z^{\varepsilon}_1 \in \partial g_j) \leq \exp (-
		\varepsilon^{- 2} (\tilde{V}  (K_i, K_j) - \gamma)), 
	\end{equation}
	for any $i, j \in \mathcal{L}$, $x \in g_i$ and $\varepsilon \in (0,
	\varepsilon_0)$.
\end{prop}

\begin{proof}
	If $\tilde{V}  (K_i, K_j) = + \infty$, then there is no continuous trajectory, which connects $K_i$ and $K_j$ without touching any $K_l$ for $l\neq i,j$. Otherwise, by Lemma \ref{lem:qpcontinuous}, the quasi-potential must be finite. Hence,
	we have $P^{\varepsilon}_x
	(Z^{\varepsilon}_1 \in \partial g_j) = 0$. Thus, inequality \eqref{transeq} holds.
	
	Now, we suppose that $\tilde{V}  (K_i, K_j) < + \infty$ and $i \neq j$.

	\textbf {Proof of the lower bound of \eqref{transeq}:}
	By Lemma \ref{lem:qponcompactset}, we can choose a
	compact set $D$ satisfying $D \supset B_0(M)$ and $\tilde{V}  (K_i, K_j) =
	\tilde{V}_D (K_i, K_j)$.  Let us arbitrarily fix $\rho_0 \in \left( 0, \delta_1 \wedge \frac{\gamma}{10 L}
	\right)$ and $\rho_1 \in (0, \rho_0)$, where $L$ is the constant
	corresponding to $D$ in Lemma \ref{lem:qpcontinuous}.
	
	By the definition of $\tilde{V}_D (K_i,
	K_j)$, there exists a pair of $T_{i j} > 0$, $\varphi^{i j} \in C ([0, T_{i
		j}] ; D)$ such that
	\[ \varphi^{i j}_0 \in K_i,\quad \varphi^{i j}_{T_{i j}} \in K_j,\quad \varphi^{i j}_t
	\in D \backslash \bigcup_{s \neq i, j} K_s \textrm{ for all } t\in[0,T_{ij}], \]
	and
	\[ S_{0 T_{i j}} (\varphi^{i j}) \leq \tilde{V} (K_i, K_j) + 0.1 \gamma . \]
	We choose positive constants $\rho_2, d_{ij}$ satisfying
	\[  \rho_2 \in \left(0, \rho_1 \wedge \frac{1}{3} \tmop{dist} \left( \varphi^{i j},
	\bigcup_{s \neq i, j} K_s \right)\right), \quad d_{i j} \in \left( 0, \frac{1}{3} \tmop{dist} \left( \varphi^{i j},
	\bigcup_{s \neq i, j} K_s \right) \wedge \rho_2 \right). \]
	
	For any $x \in \partial g_i$, by Lemma \ref{lem:qpcontinuous}, we can connect $x$ to
	$K_i$ with a $\phi^{(1)} \in C ([0, \rho_2] ; g_i \cup \partial g_i)$
	satisfying $S_{0 \rho_2} (\phi^{(1)}) \leq 0.1 \gamma$. According to Lemma
	\ref{lem:fw2}, we can construct a curve $\phi^{(2)} \in C ([0, T^{(2)}_{i
		j}] ; g_i)$ from $\phi_{\rho_2}^{(1)}$ to $\varphi^{i j}_0$ satisfying $S_{0
		T^{(2)}_{i j}} (\phi^{(2)}) \leq 0.1 \gamma$.
	
	Furthermore, we construct a curve connecting $\partial g_i$ and
	$K_j$ as follow. Firstly, we connect the end point of $\phi^{(1)}$ to
	the start point of $\phi^{(2)}$. Secondly, we connect the end point of
	$\phi^{(2)}$ to the start point of $\varphi^{i j}$. For the sake of
	simplicity, we still denote this new curve as $\varphi^{i j}$ with time
	$T_{i j}$. Hence, by \eqref{ldpl} we know that for any $ x \in \partial g_i$,
	there exists a $\varepsilon_2^{i j} > 0$ such that we have
	\begin{eqnarray*}
		P^{\varepsilon}_x (Z_1^{\varepsilon} \in \partial g_j) & \geq &
		P^{\varepsilon}_x (\rho_{0 T_{i j}} (X_{\cdot}^{\varepsilon},
		\varphi_{\cdot}^{i j}) < d_{i j})\\
		& \geq & \exp (- \varepsilon^{- 2} (S_{0 T_{i j}} (\varphi^{i j}) + 0.1
		\gamma))\\
		& \geq & \exp (- \varepsilon^{- 2} (\tilde{V}  (K_i, K_j) + 0.4 \gamma))
		,
	\end{eqnarray*}
	for any $\varepsilon \in (0, \varepsilon_1 \wedge \varepsilon_2^{i j})$.
	Thus, the lower bound of \eqref{transeq} is true.
	
	\textbf{Proof of the upper bound of \eqref{transeq}:} By the strong Markov property of the solution $X^{\varepsilon}_t$, we have
	\begin{equation}\label{upper1}
		P^{\varepsilon}_x (Z_1^{\varepsilon} \in \partial g_j) = E^{\varepsilon}_x
		\left( E^{\varepsilon}_x \left( 1_{\{
			X_{\tau_1^{\varepsilon}}^{\varepsilon} \in \partial g_j \}} |
		X^\eps_{\sigma^{\varepsilon}_1} \nobracket \right) \right) \leq
		\sup_{y \in \partial G_i} P^{\varepsilon}_y
		(X_{\tau_1^{\varepsilon}}^{\varepsilon} \in \partial g_j) , \quad \forall x \in \partial g_i.
	\end{equation}
	It is clear that for any $T^{\ast} > 0$, $y \in \partial G_i$ we have
	
	\begin{equation}\label{upper2}
		P^{\varepsilon}_y (X_{\tau_1^{\varepsilon}}^{\varepsilon} \in \partial
		g_j) \leq P^{\varepsilon}_y (\tau^{\varepsilon}_1 \geq T^{\ast}) +
		P^{\varepsilon}_y (\tau^{\varepsilon}_1 \leq T^{\ast},
		X_{\tau_1^{\varepsilon}}^{\varepsilon} \in \partial g_j) .
	\end{equation}
	By Lemma \ref{lem:uppertimeestimate}, for any $y \in \partial G_i$, there exist two
	constants $T_{i j}^{ (3)}, \varepsilon_3^{i j} > 0$ such that for any
	$\varepsilon \in (0, \varepsilon_3^{i j})$, the first term in the right hand side
	of \eqref{upper2} satisfies
	\begin{equation}\label{upper3}
		P^{\varepsilon}_y (\tau^{\varepsilon}_1 \geq T_{i j}^{(3) }) \leq \exp (-
		\varepsilon^{- 2} (\tilde{V}  (K_i, K_j) + \gamma + 1)) .
	\end{equation}
	For the second term in the left side of \eqref{upper2}, we claim that
	there exists a constant $\delta > 0$ such that for all $y \in \partial
	G_i$, we have
	\begin{equation}\label{upper4}
		\{ \varphi : \tau^{\varepsilon}_1 \leq T_{i j}^{ (3)},
		\ph_{\tau_1^{\varepsilon}}^{\varepsilon} \in \partial g_j, \ph_0^{\varepsilon}
		= y \} \nobracket \subset \{ \varphi : \rho_{0 T_{i j}^{ (3)}} (\varphi,
		\Phi_y (\tilde{V}_D (K_i, K_j) - 0.6 \gamma)) \geq \delta \}.
	\end{equation}
	Otherwise, for each $n\in \bb N^+$, we set $\delta_n=\frac{1}{n}$, then there exist three sequences: $\{ y_n \} \subset \partial G_i$,
	$\{ \varphi^{(n)} \} \subset \{ \varphi : \varphi \in C_{y_n} ([0, T_{i j}^{
		(3)}] ; \mathbb{R}^d), \tau^{\varepsilon}_1 \leq T_{i j}^{ (3)},
	\varphi_{\tau^{\varepsilon}_1} \in \partial g_j \}$ and $\{ \psi^{(n)}
	\} \subset C_{y_n} ([0, T_{i j}^{ (3)}] ; \mathbb{R}^d)$, which satisfy $S_{0 T_{i j}^{ (3)}} (\psi^{(n)}) \leq
	\tilde{V}_D (K_i, K_j) - 0.6 \gamma$ and $\rho_{0 T_{i j}^{ (3)}} \left(
	{\varphi^{(n)}} , \psi^{(n)} \right) \leq \delta_n= \frac{1}{n}$, $\forall n \in
	\mathbb{N}^+$. By Lemma
	\ref{lem:incompaceset}, there exists a compact set $\mathcal{M} \subset
	\mathbb{R}^d$ such that $ \psi^{(n)}  \subset \mathcal{M}, \forall
	n \in \mathbb{N}^+$. By some suitable choices, we can ensure that $\mathcal{M}$ with smooth $\partial \mathcal{M}$
	contains $D$. Since these $\psi^{(n)}$ may intersect with $\cup_{s \neq i, j} K_s$, we need the following
	$\tilde{\psi}^{(n)}$. By Lemma \ref{lem:fw1}, we know that there exists
	a $\beta \in \left( 0, \delta_1 \wedge \frac{\gamma}{10 L} \right)$ such
	that for any $\psi^{(n)}$, there exists a $\tilde{\psi}^{(n)}$
	satisfying $\tilde{\psi}^{(n)} \subset \mathcal{M}_{- \beta} \backslash
	\cup_{s \neq i, j} (K_s)_{\beta}$ and
	\[ S_{0 T_{i j}^{ (3)}} (\tilde{\psi}^{(n)}) \leq S_{0 T_{i j}^{ (3)}}
	(\psi^{(n)}) + 0.1 \gamma \leq \tilde{V}_D (K_i, K_j) - 0.5 \gamma .
	\]
	Let us fix a positive integer $n > \frac{10 L}{\gamma}$. By setting $T_{i j}^{ (4)}
	= \rho_1$, we can connect $K_i$ and $y_n$ through a curve $\eta^{(1)}$ satisfying
	$S_{0 T_{i j}^{ (4)}} (\eta^{(1)}) \leq 0.1 \gamma$. Furthermore, setting
	$T_{i j}^{ (5)} = \tmop{dist} \left( \tilde{\psi}_{\tau^{\varepsilon}_1
		(\varphi^{(n)})}^{n}, K_j \right)$, we can connect
	$\tilde{\psi}_{\tau^{\varepsilon}_1 (\varphi^{(n)})}^{n}$ and $K_j$
	through a curve $\eta^{(2)}$ satisfying $S_{0 T_{i j}^{ (5)}} (\eta^{(2)}) \leq
	0.2 \gamma$. We denote $\tilde{T} = T_{i j}^{ (4)} + \tau^{\varepsilon}_1
	(\varphi) + T_{i j}^{ (5)}$ and construct the curve as
	\[ \xi (t) = \left\{ \begin{array}{ll}
		\eta_t^{(1)}, \qquad  & t \in [0, T_{i j}^{ (4)}],\\
		\tilde{\psi}_t^{n}, \qquad & t \in [T_{i j}^{ (4)}, T_{i j}^{ (4)} +
		\tau^{\varepsilon}_1 (\varphi)],\\
		\eta_t^{(2)}, \qquad & t \in [T_{i j}^{ (4)} + \tau^{\varepsilon}_1
		(\varphi), \tilde{T}].
	\end{array} \right. \]
	Then we have $\xi \in C ([0, \tilde{T}] ; \mathbb{R}^d \backslash \cup_{s \neq i,
		j} (K_s)_{\beta})$, $\xi_0 \in K_i$, $\xi_{\tilde{T}} \in K_j$ and $S_{0
		\tilde{T}} (\xi) \leq \tilde{V} (K_i, K_j) - 0.1 \gamma$. The above facts
	are contradict to the definition of $\tilde{V} (K_i, K_j)$. Therefore,
	\eqref{upper4} holds.
	
	Thus, according to \eqref{ldpu}, there exists an $\varepsilon_4^{i j} >
	0$ such that
	\begin{eqnarray}\label{upper5}
		P^{\varepsilon}_y (\tau^{\varepsilon}_1 \leq T_{i j}^{ (3)},
		X_{\tau_1^{\varepsilon}}^{\varepsilon} \in \partial g_j) & \leq &
		P_y^{\varepsilon} (\{ \varphi : \rho_{0 T_{i j}^{ (3)}} (\varphi, \Phi_y
		(\tilde{V}_D (K_i, K_j) - 0.6 \gamma)) \geq \delta \}) \nonumber\\
		& \leq & \exp (- \varepsilon^{- 2} (\tilde{V}  (K_i, K_j) - 0.7 \gamma))
	\end{eqnarray}
	for any $\varepsilon \in (0, \varepsilon_4^{i j})$ and $y \in \partial G_i$.
	
	Combining \eqref{upper1}-\eqref{upper3} with \eqref{upper5}, for
	any $x \in \partial g_i$, $\gamma > 0$ as long as $\varepsilon <
	\varepsilon_1 \wedge \varepsilon_3^{i j} \wedge \varepsilon_4^{i j}$ we have
	\[ P^{\varepsilon}_x (X_{\tau_1^{\varepsilon}}^{\varepsilon} \in \partial
	g_j) \leq \exp (- \varepsilon^{- 2} (\tilde{V}  (K_i, K_j) - \gamma)) .
	\]
	Thus, the proof of the case $\tilde{V}  (K_i, K_j) < + \infty$, $i \neq j$
	is finished.
	
	For the case $i=j \in \cl L$, we have $\tilde{V}  (K_i, K_i) =0$. Thus, the upper bound is
	clear. 
	The lower bound can be proved by the same method as in the case of $i \neq j$.
	
	Finally, let us set $\varepsilon_2 = \min_{i, j} \varepsilon_2^{i j}$, $\varepsilon_3 =
	\min_{i, j} \varepsilon_3^{i j}$ and $\varepsilon_4 = \min_{i, j}
	\varepsilon_4^{i j}$. From the above discussion, if
	we choose $\rho_0 < \delta_1 \wedge \frac{\gamma}{10 L}$, $\rho_1 < \rho_0$,
	$\rho_2 <  \rho_1 \wedge \frac{1}{3} \tmop{dist} \left( \varphi^{i j},
	\cup_{s \neq i, j} K_s \right)$ and $\varepsilon_0 <
	\varepsilon_1 \wedge \varepsilon_2 \wedge \varepsilon_3 \wedge
	\varepsilon_4$, then \eqref{transeq} holds for all $i, j \in \mathcal{L}$, $x \in g_i$ and $\varepsilon \in (0,
	\varepsilon_0)$.
	
	Therefore, the proof of Proposition \ref{prop:transpb} is complete.
	
\end{proof}

For the rest of our paper, we need following notations and results from {\cite{fw}}.

For any
$i\in \mathcal{L}$, $m \in \mathcal{L} \backslash \{i\}$, $n \in
\mathcal{L}$ and $m \neq n$, a set consisting of arrows ``$m \rightarrow
n$'' is called an $\{i\}$-graph if
\begin{trivlist}
	\item [i).]
	Every $m \in \mathcal{L} \backslash \{i\}$ is an initial point of exactly
	one arrow.
	
	\item[ii).]
	There are no cycles in the set.
\end{trivlist}
Let $G (i)$ be the set of all $\{i\}$-graphs
and set
\[ W (K_i) = \min_{q \in G (i)} \sum_{(m \rightarrow n) \in q} \tilde{V} (K_m, K_n) . \]
$W (K_i)$ can be understood as the minimum total cost of getting $K_i$. 

Lemma 6.4.1 in {\cite{fw}} shows $W (K_i) = \min_{q \in G (i)}
\sum_{(m \rightarrow n) \in q} V (K_m, K_n)$.

According to the proof of lemma 6.4.3 in {\cite{fw}},
it is easy to check that $\min_{i \in \mathcal{L}} W
(K_i)$ can only be attained in the stable $K_i$ for \eqref{1.2}, since all $\omega$-limit sets of \eqref{1.1} are in a compact set $\cup_{i \in \mathcal{L}}K_i $.

\begin{lem}\label{lem:invariantmeasureestimate}
	Under conditions i) - iii) of Proposition \ref{prop:LDP} and 1)
	of Assumption \ref{mainassumption}, for any $\gamma > 0$, there exist positive constants $\rho_2, \varepsilon_0 $ such that for any $\varepsilon \in (0, \varepsilon_0)$,
	the invariant measure $\nu_{\varepsilon}$ of $\{ Z^{\varepsilon}_n \}$
	satisfies
	\begin{equation}\label{invariantmeasofMC}
		\nu_{\varepsilon} (\partial g_i) \in (e (i, \varepsilon, 4 (l - 1)
		\gamma), e (i, \varepsilon, - 4 (l - 1) \gamma)), 
	\end{equation}
	where $e (i, \varepsilon, \gamma): = \exp (- \varepsilon^{- 2} (W
	(K_i) - \min_j W (K_j) + \gamma))$.
\end{lem}

\begin{proof}
	By Proposition \ref{prop:transpb}, for any $\gamma > 0$, there exist
	$\rho_0, \rho_1, \rho_2 > 0$ and an $\varepsilon_0 > 0$ such that
	\eqref{transeq} holds for any $\varepsilon \in (0, \varepsilon_0)$ and
	every pair of $i, j$. Thus, according to the standard result in Lemma 6.3.2 of
	{\cite{fw}}, we know that \eqref{invariantmeasofMC} is true.
	
\end{proof}

\begin{lem}\label{lem:invariantmeasureexist}
	Under conditions i) - ii) of Proposition \ref{prop:LDP} and 1), 3) 
	of Assumption \ref{mainassumption}, there exist two positive constants $\delta_2, \delta_3$
	such that for any fixed $\rho_1 \in (0,
	\delta_2)$ and $ \rho_2\in (0, \rho_1 \wedge \delta_3 )$, there exists an $\varepsilon_5=\varepsilon_5(\rho_2)>0$ such that for any $A \subset \mathcal{B} (\mathbb{R}^d)$, the unique invariant
	measure $\mu_{\varepsilon}$ of \eqref{1.2} can be represented as
	\begin{equation}\label{invariantmeasurebounded}
		\mu_{\varepsilon} (A) = \int_{\partial g} E^{\varepsilon}_y
		\int^{\tau^{\varepsilon}_1}_0 1_A (X^{\varepsilon}_t) d t
		\nu_{\varepsilon} (d y),
	\end{equation}
	for any $\varepsilon \in (0,\varepsilon_5)$.
\end{lem}

\begin{proof}
	The uniqueness of $\mu_{\varepsilon}$ is owing to the strong Feller property
	and irreducibility of $X^{\varepsilon}_t$, which can be found in
	{\cite{dyl}} and {\cite{dp}}. Thus, we only need to prove that the measure
	given in \eqref{invariantmeasurebounded} is a finite invariant measure. To show
	the finiteness of $\mu_\varepsilon$, we should prove
	$\mu_{\varepsilon} (\mathbb{R}^d) < + \infty$. In fact,
	\begin{eqnarray}\label{invariantmeasurebounded1}
		\mu_{\varepsilon} (\mathbb{R}^d) & = & \int_{\partial g} E^{\varepsilon}_y
		\int^{\tau^{\varepsilon}_1}_0 1_{\mathbb{R}^d} (X^{\varepsilon}_t) d t
		\nu_{\varepsilon} (d y) \nonumber\\
		& = & \int_{\partial g} E^{\varepsilon}_y \tau^{\varepsilon}_1
		\nu_{\varepsilon} (d y) \nonumber\\
		& \leq & \sup_{y \in \partial g} E^{\varepsilon}_y \tau^{\varepsilon}_1 .
	\end{eqnarray}
	By the strong Markov property of $X^{\varepsilon}_t$, we have
	\begin{equation}\label{invariantmeasurebounded2}
		E^{\varepsilon}_y \tau^{\varepsilon}_1 = E^{\varepsilon}_y
		(E^{\varepsilon}_y (\tau^{\varepsilon}_1 - \sigma_1^{\varepsilon} |
		\nobracket \mathcal{F}_{\sigma_1^{\varepsilon}}) + \sigma_1^{\varepsilon})
		\leq \sup_{z \in \partial G} E^{\varepsilon}_z \tau^{\varepsilon}_1 +
		E^{\varepsilon}_y \sigma_1^{\varepsilon}.
	\end{equation}
	
	According to Lemma 6.1.7 and Lemma 6.1.8 in {\cite{fw}}, we know that for any $\gamma>0$, there exist
	$\delta_2, \delta_3, \varepsilon' \in (0, + \infty)$ such that for any $\rho_1 \in (0,
	\delta_2)$, $\rho_2 \in (0,\delta_3 \wedge \rho_1) $ and $\varepsilon \in (0, \varepsilon')$ we have
	\begin{equation}\label{recurrentexpectation1}
		E^{\varepsilon}_y \sigma_1^{\varepsilon} \in \exp \left( \pm \varepsilon^{-
			2} \frac{\gamma}{2 l} \right).
	\end{equation}
	On the other hand, according to 3) of Assumption \ref{mainassumption} and Theorem 3.9 in {\cite{km}} (We take $V(s, x)$ therein as $U (x)$, $U_1$ therein as $\mathbb{R}^d \backslash g$),
	there exists a $\beta=\beta(\rho_1,\rho_2) \in (0, + \infty)$, such that
	\begin{equation}\label{recurrentexpectation2}
		\sup_{z \in \partial G} E^{\varepsilon}_z \tau^{\varepsilon}_1 \leq \beta \triangleq \chi_{\rho_2}^{-1}\max_{x \in \partial G}\tilde{U}(x),
	\end{equation}
	for any $\varepsilon \in (0, \varepsilon_{\rho_2})$. Let 
	$\varepsilon_5 = \varepsilon_{\rho_2} \wedge \varepsilon'$.
	By \eqref{invariantmeasurebounded1} and \eqref{invariantmeasurebounded2}, we get
	\[ E^{\varepsilon}_y \tau^{\varepsilon}_1 < \beta + \exp \left( -
	\varepsilon^{- 2} \frac{\gamma}{2 l} \right) < + \infty, \]
	for any $y \in \partial g, \varepsilon \in (0,\varepsilon_5)$, which implies the well-posedness of $\mu_{\varepsilon}$.
	
	To show $\mu_{\varepsilon}$ is an invariant measure of
	$X_t^{\varepsilon}$, we need to prove that for any bounded continuous function $f$ in
	$\mathbb{R}^d$, $\mu_{\varepsilon}$ satisfies
	\begin{equation*}
		\int_{\mathbb{R}^d} f (x) \mu_{\varepsilon} (d x) = \int_{\mathbb{R}^d}
		E^{\varepsilon}_x f (X^{\varepsilon}_t) \mu_{\varepsilon} (d x), \quad
		\forall t \in (0, + \infty).
	\end{equation*}
	This can be proved by the method provided in Theorem 4.1 in
	{\cite{km}}.
	
\end{proof}

\begin{rem}\label{simplifiedbeta}
	If $U(x)$ satisfies 3)$'$, then we can choose
	\begin{equation*}
		\beta = 2 (\min_{x \in \mathbb{R}^d \backslash g} \zeta| \nabla U (x) |^2
		\wedge \chi)^{- 1} \max_{x \in \partial G} U (x).
	\end{equation*}
	This fact means that for any fixed $ \rho_1, \rho_2>0 $, there exists a constant $\varepsilon_{\rho_2}>0$ satisfying $\sup_{z \in \partial G} E^{\varepsilon}_z \tau^{\varepsilon}_1 \leq \beta$, for any $\varepsilon \in (0,\varepsilon_{\rho_2})$.
\end{rem}

\

The next theorem is the main result of this paper. Let $I = \{ i : i
\in \mathcal{L}, K_i\ \tmop{is}\ a \ \tmop{stable}\ \tmop{set} \}$ and its subset
$I_0 = \{ i : i \in I, W (K_i) = \min_{j \in \mathcal{L}} W (K_j)
\}$.

\begin{thm}\label{thm:mian}
	Suppose that conditions i) - ii) of Proposition \ref{prop:LDP} and
	Assumption \ref{mainassumption} are true, then $\mu$ supports on
	$\cup_{i \in I_0} K_i$.
\end{thm}

\begin{proof}
	For the sake of simplicity of notations, we write \eqref{con23} of Assumption \ref{mainassumption} in an equivalent way: For any stable $K_i$,
	there exist three common constants $\tilde{\delta} > 0, k_1 \geq 1, k_2>0 $ such that for any $\delta', \delta''$ admitting $0 <
	k_1^{\frac{1}{2}} \delta'' < \delta' < \tilde{\delta}$, we have 
	\begin{equation}\label{equivalent}
		\min_{x \in \partial (K_i)_{\delta'}, y \in \partial (K_i)_{\delta''}} U
		(y) - U (x) \geq k_2\left( (\delta')^2 - k_1 (\delta'')^2\right).
	\end{equation}
	
	Let $\Delta_{i j} = V (K_i, K_j)$ and $\Delta = \min_{i
		\in I, j \in \mathcal{L}} \Delta_{i j}$. $\Delta_{i j}
	> 0$ is because of the definition of the stable set. For every $K_i, i \in
	\mathcal{L}$, by the continuity of $V (x, y)$ in $\{ x : | x | < M
	\}$, there exists a $\delta_4 > 0$ such that for any $i \in \mathcal{L}$ and
	$x \in (K_i)_{\delta_4}$, we have $V (x, K_i) < \frac{1}{8} \Delta$ and $V
	(K_i, x) < \frac{1}{8} \Delta$.
	
	Let us fix a $\delta \in \left( 0, \frac{1}{3} (\delta_4 \wedge
	\tilde{\delta} \wedge \delta_1) \right)$. For some $\rho_1 \in (0,\delta)$ and some constant $c\in(0,1)$, we want to have:
	$$ \min_{x \in \partial (G_i)_{c\delta},\\ y \in \partial (G_i)_{\delta}} (U (y) - U (x)) > 0. $$
	Therefore, according to 2) of Assumption \ref{mainassumption}, we constrain $c<k_1^{-\frac{1}{2}}$ and $ \rho_1 < \frac{1-c}{2}\delta$.
	
	Under above constraints, let us divide
	$\mathbb{R}^d$ into $l + 1$ parts: $(G_1)_{\delta}$, $(G_2)_{\delta}$,
	{\textdots}, $(G_l)_{\delta}$ and $\mathbb{R}^d \backslash (G)_{\delta}$.
	For all stable $K_i$, let us denote $\Upsilon_i = \min_{x \in 
		G_i, y \in \partial (G_i)_{\delta}} V (x, y)$.
	$\Upsilon_i > 0$ is owing to the definition of the stable set. By 2) of Assumption \ref{mainassumption} and Lemma \ref{lem:lowerboundedofqp}, we have
	\begin{equation}\label{lowerenergyestimatationofstableset}
		\min_{i \in I} \Upsilon_i \geq \min_{x \in (G_i)_{c\delta}, y \in \partial (G_i)_{\delta}} \frac{2 \zeta}{{\bar{\lmd}}^2}(U (y) - U (x)) > \frac{2 \zeta k_2}{{\bar{\lmd}}^2}\left[ \left( \frac{3-c}{2}\right)^2 -k_1\left(\frac{1+c}{2} \right)^2 \right]\delta^2, 
	\end{equation}
	for any $\rho_1 \in (0, \frac{1-c}{2}\delta)$. We set 
	$$ \Upsilon = \frac{2 \zeta k_2}{{\bar{\lmd}}^2}\left[ \left( \frac{3-c}{2}\right)^2 -k_1\left(\frac{1+c}{2} \right)^2 \right]\delta^2
	$$
	and $\check{W} = \min_{i \nin I_0} W (K_i) - \min_{i \in \mathcal{L}} W (K_i)$.
	Let us choose an arbitrary fixed $\gamma \in \left( 0, \frac{1}{5} (
	\Upsilon\wedge \Delta \wedge \check{W}) \right)$.
	
	By Lemma \ref{lem:invariantmeasureestimate} and Lemma \ref{lem:invariantmeasureexist}, there exist proper
	$\rho_0, \rho_1, \rho_2, \varepsilon_0 $ and $\varepsilon_5$ satsifying $ \rho_1 < \frac{1-c}{2}\delta \wedge \delta_2, \rho_2 < \delta_3 $ and other restrictions of Proposition \ref{prop:transpb} such that we have
	\begin{equation}\label{invariantmeasureofz_n}
		\nu_{\varepsilon} (\partial g_i) \in \left( e \left( i, \varepsilon,
		\frac{l - 1}{l} \gamma \right), e \left( i, \varepsilon, - \frac{l - 1}{l}
		\gamma \right) \right),
	\end{equation}
	and the well-posedness of $\mu_{\varepsilon}$,
	for any $\varepsilon \in (0, \varepsilon_0 \wedge \varepsilon_5)$.
	
	Next, under the above restrictions, we will give some estimates of $\mu_\eps((G_i)_{\delta})$, $i\in\cl L$ and $\mu_\eps(\mathbb{R}^d \backslash (G)_{\delta})$:
	
	\textbf{Proof of the estimate of $ \mu_{\varepsilon} ((G_i)_{\delta}) $:}
	According to \eqref{invariantmeasurebounded}, because $\rho_1, \delta < \delta_1 $, we have 
	\begin{eqnarray}\label{invariantmeasureestimationofG_i}
		\mu_{\varepsilon} ((G_i)_{\delta}) & = & \int_{\partial g}
		E^{\varepsilon}_y \left(\int^{\tau^{\varepsilon}_1}_0 1_{(G_i)_{\delta}}
		(X^{\varepsilon}_t) d t\right) \nu_{\varepsilon} (d y) \nonumber\\
		& = & \int_{\partial g_i}\left( E^{\varepsilon}_y \sigma^{\varepsilon}_1\right)
		\nu_{\varepsilon} (d y) + \sum_{j = 1}^l \int_{\partial g_j}
		E^{\varepsilon}_y \left(\int^{\tau^{\varepsilon}_1}_{\sigma^{\varepsilon}_1}
		1_{(G_i)_{\delta}} (X^{\varepsilon}_t) d t\right) \nu_{\varepsilon} (d y),
	\end{eqnarray}
	for any $i\in \mathcal{L} $.
	By \eqref{invariantmeasureofz_n}, \eqref{recurrentexpectation1} and \eqref {recurrentexpectation2}, we have
	\begin{equation}\label{invariantmeasureestimationofG_i1}
		\int_{\partial g_i} \left(E^{\varepsilon}_y \sigma^{\varepsilon}_1\right)
		\nu_{\varepsilon} (d y) \in \left( e \left( i, \varepsilon, \frac{l -
			0.5}{l} \gamma \right), e \left( i, \varepsilon, - \frac{l - 0.5}{l}
		\gamma \right) \right)
	\end{equation}
	and $\max_{z \in \partial G}
	E^{\varepsilon}_z \tau^{\varepsilon}_1 \leq \beta$, for any $\varepsilon \in (0, \varepsilon_0 \wedge \varepsilon_5)$.
	
	For $j = i$, by the strong Markov property, we have
	\begin{eqnarray}\label{invariantmeasureestimationofG_i2}
		\int_{\partial g_i} E^{\varepsilon}_y
		\left(\int^{\tau^{\varepsilon}_1}_{\sigma^{\varepsilon}_1} 1_{(G_i)_{\delta}}
		(X^{\varepsilon}_t) d t\right) \nu_{\varepsilon} (d y) & = & \int_{\partial g_i}
		E^{\varepsilon}_y
		E_{X^{\varepsilon}_{\sigma^{\varepsilon}_1}}^{\varepsilon} \left(
		\int^{\tau^{\varepsilon}_1}_0 1_{(G_i)_{\delta}} (X^{\varepsilon}_t) d t
		\right) \nu_{\varepsilon} (d y) \nonumber\\
		& \leq & \max_{z \in \partial G_i} E^{\varepsilon}_z \tau^{\varepsilon}_1
		\cdot \nu_{\varepsilon} (\partial g_i) \nonumber\\
		& \leq & \beta e \left( i, \varepsilon, - \frac{l - 1}{l}
		\gamma \right),
	\end{eqnarray}
	for any $\varepsilon \in ( 0 , \varepsilon_5)$.
	
	For $j \neq i$, by the strong Markov property, we have
	\begin{eqnarray}\label{invariantmeasureestimationofG_i3}
		\int_{\partial g_j} E^{\varepsilon}_y
		\left(\int^{\tau^{\varepsilon}_1}_{\sigma^{\varepsilon}_1} 1_{(G_i)_{\delta}}
		(X^{\varepsilon}_t) d t\right) \nu_{\varepsilon} (d y) & = & \int_{\partial g_j}
		E^{\varepsilon}_y
		E_{X^{\varepsilon}_{\sigma^{\varepsilon}_1}}^{\varepsilon} \left(
		\int^{\tau^{\varepsilon}_1}_0 1_{(G_i)_{\delta}} (X^{\varepsilon}_t) d t
		\right) \nu_{\varepsilon} (d y) \nonumber\\
		& \leq & \max_{z \in \partial G_j} P^{\varepsilon}_z
		(\tau^{\varepsilon}_{(G_i)_{\delta}} < \tau^{\varepsilon}_1) \cdot \max_{z
			\in \partial G_j} E^{\varepsilon}_z \tau^{\varepsilon}_1 \cdot
		\nu_{\varepsilon} (\partial g_j) .
	\end{eqnarray}
	For any $z \in \partial G_j$ and $T > 0$, we have
	\begin{equation*}\label{invariantmeasureestimationofG_i4}
		P^{\varepsilon}_z (\tau^{\varepsilon}_{(G_i)_{\delta}} <
		\tau^{\varepsilon}_1) \leq P^{\varepsilon}_z
		(\tau^{\varepsilon}_{(G_i)_{\delta}} < \tau^{\varepsilon}_1 \leq T) +
		P^{\varepsilon}_z (\tau^{\varepsilon}_1 > T) .
	\end{equation*}
	According to Lemma \ref{lem:uppertimeestimate}, there exist $T_0, \varepsilon_6 > 0$ such
	that we have 
	\begin{equation}\label{eq:PGtau>T}
		P^{\varepsilon}_z (\tau^{\varepsilon}_1 > T_0) < \exp (-
		\varepsilon^{- 2} (W (K_j) - \min_i W (K_i) + \Delta))
	\end{equation}
	for any $z \in
	\partial G_j $ and $ \varepsilon \in (0, \varepsilon_6) $.
	
	Suppose that $K_j$ is a stable set,
	for the upper bound of $P^{\varepsilon}_z
	(\tau^{\varepsilon}_{(G_i)_{\delta}} < \tau^{\varepsilon}_1 \leq T_0)$, according to the choice of $\delta$  and \eqref{ldpu}, there exist $\theta_{ij}, \varepsilon^{ij}_7 > 0$ such that we have
	\begin{equation}\label{invariantmeasureestimationofG_i5}
		P^{\varepsilon}_z (\tau^{\varepsilon}_{(G_i)_{\delta}} <
		\tau^{\varepsilon}_1 \leq T_0) \leq P^{\varepsilon}_z (\{\ph:\rho_{0 T_0}
		(\ph, \Phi_z (0.5 \Delta)) > \theta_{ij}\}) \leq \exp (-
		\varepsilon^{- 2} (0.4 \Delta)),
	\end{equation}
	for any $z \in \partial G_j $ and $ \varepsilon \in (0, \varepsilon^{i j}_7)$.
	The existence of $\theta_{ij}$ can be proved by contradiction as the proof of \eqref{upper4}. We denote $\varepsilon_7 = \min_{i,j} \varepsilon_7^{i,j}$.
	Then by \eqref{invariantmeasureestimationofG_i3}, \eqref{eq:PGtau>T} and \eqref{invariantmeasureestimationofG_i5} we have
	\begin{equation}\label{invariantmeasureestimationofG_i6}
		\int_{\partial g_j} E^{\varepsilon}_y
		\left(\int^{\tau^{\varepsilon}_1}_{\sigma^{\varepsilon}_1} 1_{(G_i)_{\delta}}
		(X^{\varepsilon}_t) d t\right) \nu_{\varepsilon} (d y) < 2\beta \exp \left( -
		\varepsilon^{- 2} \left( W (K_j) - \min_i W (K_i) - \frac{l - 1}{l} \gamma
		+ 0.4 \Delta \right) \right), 
	\end{equation}
	for $j \neq i$ and $j \in I$.
	
	If $K_j$ is an unstable set, then by \eqref{invariantmeasureestimationofG_i3}, we have
	\begin{equation}\label{invariantmeasureestimationofG_i7}
		\int_{\partial g_j} E^{\varepsilon}_y
		\left(\int^{\tau^{\varepsilon}_1}_{\sigma^{\varepsilon}_1} 1_{(G_i)_{\delta}}
		(X^{\varepsilon}_t) d t \right) \nu_{\varepsilon} (d y) \leq \beta \exp \left( -
		\varepsilon^{- 2} \left( W (K_j) - \min_i W (K_i) - \frac{l - 1}{l} \gamma
		\right) \right),
	\end{equation}
	since  $P^{\varepsilon}_z(\tau^{\varepsilon}_{(G_i)_{\delta}} < \tau^{\varepsilon}_1)\le 1$.
	Equations \eqref{invariantmeasureestimationofG_i1}, \eqref{invariantmeasureestimationofG_i6} and \eqref{invariantmeasureestimationofG_i7} tell us that
	\begin{equation}\label{invariantmeasureestimationofG_i8}
		\mu_{\varepsilon} ((G_i)_{\delta}) \in (\exp (- \varepsilon^{- 2} \gamma),
		\exp (\varepsilon^{- 2} \gamma)),
	\end{equation}
	for any $K_i$, $i \in I_0$ and 
	\begin{equation}\label{invariantmeasureestimationofG_i9}
		\mu_{\varepsilon} ((G_i)_{\delta}) \leq \exp (- \varepsilon^{- 2}
		(\check{W} \wedge 0.4 \Delta - \gamma)),
	\end{equation}
	for any $K_i$, $i \in \mathcal{L} \backslash I_0$,   $\varepsilon \in (0, \varepsilon_0 \wedge \varepsilon_5 \wedge \varepsilon_6 \wedge \varepsilon_7) $.
	
	\textbf{Proof of the upper bound of $\mu_{\varepsilon}(\mathbb{R}^d
		\backslash (G)_{\delta})$:} In fact, we have
	\begin{eqnarray}\label{invariantmeasureestimationofoutside1}
		\mu_{\varepsilon} (\mathbb{R}^d \backslash (G)_{\delta}) & = &
		\int_{\partial g} E^{\varepsilon}_y \left(\int^{\tau^{\varepsilon}_1}_0 1_{\{
			\mathbb{R}^d \backslash (G)_{\delta} \}} (X^{\varepsilon}_t) d t\right)
		\nu_{\varepsilon} (d y) \nonumber\\
		& = & \sum_{i = 1}^l \int_{\partial g_i} E^{\varepsilon}_y
		E_{X^{\varepsilon}_{\sigma^{\varepsilon}_1}}^{\varepsilon} \left(
		\int^{\tau^{\varepsilon}_1}_0 1_{\{ \mathbb{R}^d \backslash (G)_{\delta}
			\}} (X^{\varepsilon}_t) d t \right) \nu_{\varepsilon} (d y) \nonumber\\
		& \leq & \sum_{i = 1}^l \max_{z \in \partial G_i} P^{\varepsilon}_z
		(\tau^{\varepsilon}_{\{ \mathbb{R}^d \backslash (G)_{\delta} \}} <
		\tau^{\varepsilon}_1) \cdot \max_{z \in \partial G_i} E^{\varepsilon}_z
		\tau^{\varepsilon}_1 \cdot \nu_{\varepsilon} (\partial g_i) .
	\end{eqnarray}
	Just like above, for any stable $K_i$, according to
	\eqref{lowerenergyestimatationofstableset}, Lemma \ref{lem:uppertimeestimate} and \eqref{ldpu}, there exist
	$\varepsilon_8^i, \theta_{i}^{'}, T > 0$ such that we have
	\begin{eqnarray}\label{invariantmeasureestimationofoutside2}
		P^{\varepsilon}_z (\tau^{\varepsilon}_{\{ \mathbb{R}^d \backslash
			(G)_{\delta} \}} < \tau^{\varepsilon}_1) & \leq & P^{\varepsilon}_z
		(\tau^{\varepsilon}_{\{ \mathbb{R}^d \backslash
			(G)_{\delta} \}} < \tau^{\varepsilon}_1 \leq T) + P^{\varepsilon}_z ( \tau^{\varepsilon}_1 > T)
		\nonumber\\
		& \leq & P^{\varepsilon}_z
		(\{\ph:\rho_{0 T} (\ph, \Phi_z (0.5 \Upsilon )) \geq \theta_i^{'}\}) + P^{\varepsilon}_z ( \tau^{\varepsilon}_1 > T)
		\nonumber\\
		& \leq & \exp (- \varepsilon^{- 2} (0.4 \Upsilon)) ,
	\end{eqnarray}
	for any $\varepsilon \in (0,\varepsilon_8^i)$ and $z \in \partial G_i$. 
	Moreover, for any unstable $K_i$ and $z \in \partial G_i$, we use
	$P^{\varepsilon}_z (\tau^{\varepsilon}_{\{ \mathbb{R}^d \backslash
		(G)_{\delta} \}} < \tau^{\varepsilon}_1) \leq 1$ directly. We denote $\varepsilon_8 = \min_{i\in I} \varepsilon_8^i$.
	
	Hence,
	\eqref{invariantmeasureestimationofoutside1} and \eqref{invariantmeasureestimationofoutside2} show 
	\begin{equation}\label{invariantmeasureestimationofoutside3}
		\mu_{\varepsilon} (\mathbb{R}^d \backslash (G)_{\delta}) \leq l \beta
		\exp (- \varepsilon^{- 2} (0.4 \Upsilon  \wedge \check{W} -
		\gamma)),
	\end{equation}
	for any $\varepsilon \in (0, \varepsilon_0 \wedge \varepsilon_5 \wedge \varepsilon_8)$.
	
	Finally, because $\gamma \in \left( 0, \frac{1}{5} ( \Upsilon \wedge \Delta \wedge
	\check{W}) \right)$, equations \eqref{invariantmeasureestimationofG_i8}, \eqref{invariantmeasureestimationofG_i9} and \eqref{invariantmeasureestimationofoutside3} show that $\mu_\eps$
	concentrates on $(G_i)_{\delta}$, $i \in I_0$ as $\varepsilon \downarrow 0$.
	Because we can choose $\delta$ and $\rho_1$ arbitrarily small, we have proven that $\mu$
	supports on $\cup_{i \in I_0} K_i$.
	
\end{proof}

\begin{rem}\label{rem:order}
	In the above proof, the order of choosing constants is extremely important. We firstly settle down $\delta$, secondly $\gamma$, then $\rho_0,\rho_1$ and $ \rho_2$, finally $\varepsilon$.
\end{rem}

According to i) in Proposition \ref{prop:LDP}, \eqref{1.1} can be restricted in $B_0^c(M)$. Thus, by Theorem 2.1 in \cite{chen}, $\mu$ is an invariant measure for \eqref{1.1} defined on $B_0^c(M)$ and supports on $\{x\in\mathbb{R}^d:x\in\omega(x)\}$. Here for \eqref{1.1}, $\omega(x)$ represents the $\omega$ limit set of $x$. Therefore, we have the following corollary. 

\begin{cor}\label{co:cap}
	Under conditions of Theorem \ref{thm:mian}, $\mu$ supports on 
	$\cup_{i \in I_0} K_i \cap \{x\in\mathbb{R}^d:x\in\omega(x)\}$.
\end{cor}

An important application of Theorem \ref{thm:mian} is the quasi-gradient
system with stochastic perturbations. Let us consider the following SDE:
\begin{equation}\label{qgsystem}
	\left\{ \begin{array}{cl}
		d X^{\varepsilon}_t &= (- \nabla U (X^{\varepsilon}_t) + H
		(X^{\varepsilon}_t)) d t + \varepsilon d W_t,\\
		X^{\varepsilon}_0 &= x_0 .
	\end{array} \right. 
\end{equation}
Here we need three conditions: $U \in C^2 (\mathbb{R}^d ; \mathbb{R}), H \in C \left( {\mathbb{R}^d}  ;
\mathbb{R}^d \right)$ and $\left\langle \nabla U (x), H (x)\right\rangle=0$ for any $x \in
\mathbb{R}^d$. The next corollary comes from Theorem \ref{thm:mian}.

\begin{cor}\label{co:qgsystem}
	Suppose that following conditions are true:
	\begin{trivlist}
		\item[i).] There exist two positive constants $M,
		\kappa$ such that
		
		$$\nabla U(x) \cdot \frac{x}{|x|} \geq \kappa,\quad
		\forall x \in B_0^c(M).$$ 
		
		\item[ii).] There is a finite number of compact
		equivalent sets $K_1, K_2, \ldots, K_l$, all of which are contained in $B_0(M)$, satisfying $x \nsim y$ for $x \in K_i, y \in K^c_i, \forall \ i \in
		\mathcal{L}$. Furthermore, every $\omega$-limit set of \eqref{1.1} is
		contained entirely in one of $K_i,\; i \in \mathcal{L}$.
		
		\item[iii).]
		For any $x \in \mathbb{R}^d \backslash \cup_{i \in \mathcal{L}}
		K_i$, $\nabla U (x) \neq 0$. 
		Moreover, 
		there exists a constant $\tilde{\delta} > 0$ and two constants
		$k_1 \ge k_2>0 $  such that for any $\delta', \delta''$ with
		$0 < \left(\frac{k_1}{k_2}\right)^{\frac{1}{2} }\delta'' < \delta' < \tilde{\delta}$,  we have
		\begin{equation}\nonumber
			\min_{x \in \partial (K_i)_{\delta''}, y \in \partial (K_i)_{\delta'}} U
			(y) - U (x) \geq k_2 (\delta')^2 - k_1 (\delta'')^2,
		\end{equation}
		for any stable $K_i, i \in \mathcal{L}$.
		\item[iv).] There exist two positive constants $\chi$ and $\varepsilon_1$ satisfying
		\begin{equation}\nonumber
			\frac{\varepsilon^2}{2} \Delta U(x) -\zeta \left|\nabla U(x) \right|^2 < -
			\chi,
		\end{equation}
		for any $\varepsilon\in \left( 0, \varepsilon_1\right) $ and $ x \in B_0^c(M)$. 
	\end{trivlist}

	Then $\mu$ of \eqref{qgsystem}
	supports on $\cup_{i \in I_0} K_i$.
\end{cor}

\medskip

\section{The large deviations of $\{\mu_\eps\}$}

In this section, we study the rate of convergence of $\{\mu_\eps\}$. We first proof the LDP of $\{\mu_\eps\}$ and then give an estimate of the rate function. In this chapter, we assume the conditions of Theorem \ref{thm:mian} are satisfied.   

\begin{thm}\label{ldp}
	$\{\mu_\eps\}$ satisfies the large deviations principle with good rate function $\mathcal{S}(x)= \min_{i \in \mathcal{L}}(W(K_i)+V(K_i,x))-\min_{i \in \mathcal{L}} W(K_i)$ as $\varepsilon \downarrow 0$.
\end{thm}

\begin{proof}
	Before the proof, we mention that for any $x\in\mathbb{R}^d$, it can be viewed as an equivalent set $K_{l+1}$. According to Lemma 6.4.3 in \cite{fw}, we have
	$$W(K_{l+1})=\min_{i \in \mathcal{L}}(W(K_i)+V(K_i,x)).$$
	We prove this theorem in three steps.
	
	First, we prove the level set of $\mathcal{S}(x)$ is compact. According to the definition of $\mathcal{S}$ and the continuity of quasi-potential,  $ \mathcal{S}(x) $ is a continuous function. Moreover, for any $\alpha \in (0, +\infty)$, we have $ \{ \mathcal{S}(x) \leq \alpha \} \subset \cup_{i \in \mathcal{L}} \{ x: V(K_i,x) \leq \alpha \}$. Thus, by \eqref{ash6}, Lemma \ref{lem:lowerboundedofqp} and the continuity of $\mathcal{S}(x)$, the level sets of $\mathcal{S}$ is compact.
	
	Second, we prove the lower estimate of LDP. For any fixed $ x \in \mathbb{R}^d$ and $\delta, \gamma \in (0,+\infty)$, according to the proof of \eqref{invariantmeasureestimationofG_i1}, there exist proper $\rho_2,\varepsilon_0 \in (0, +\infty)$ such that we have $g_{l+1} \subset (K_{l+1})_\delta$ and 
	\begin{eqnarray}\label{mulowerldp}
		\mu_{\varepsilon}((x)_\delta) & \geq & \mu_{\varepsilon} (g_{l+1})\nonumber \\
		& = & \int_{\partial g}
		E^{\varepsilon}_y \left(\int^{\tau^{\varepsilon}_1}_0 1_{g_{l+1}}
		(X^{\varepsilon}_t) d t\right) \nu_{\varepsilon} (d y) \nonumber\\
		& = & \int_{\partial g_{l+1}} \left( E^{\varepsilon}_y \sigma^{\varepsilon}_1\right)
		\nu_{\varepsilon} (d y) + \sum_{j = 1}^{l+1} \int_{\partial g_j}
		E^{\varepsilon}_y \left(\int^{\tau^{\varepsilon}_1}_{\sigma^{\varepsilon}_1}
		1_{g_{l+1}} (X^{\varepsilon}_t) d t\right) \nu_{\varepsilon} (d y)\nonumber\\
		& = & \int_{\partial g_{l+1}} \left( E^{\varepsilon}_y \sigma^{\varepsilon}_1 \right)
		\nu_{\varepsilon} (d y) \nonumber\\
		& \geq & \exp\left(-\varepsilon^{-2} \left(\min_{i \in \mathcal{L}}\left(W(K_i)+V(K_i,x)\right)-\min_{i \in \mathcal{L}} W(K_i)+ \frac{l + 0.5}{l+1} \gamma \right) \right)\nonumber\\
		& \geq & \exp(-\varepsilon^{-2}(\mathcal{S}(x)+\gamma)), 
	\end{eqnarray}
	for any $\varepsilon \in (0,\varepsilon_0)$.

	Finally, we need to prove for any fixed $ x \in \mathbb{R}^d$,  $ s\in[0,+\infty) $ and $\delta, \gamma \in (0,+\infty)$, there exists an $\varepsilon_0 \in (0, +\infty)$ satisfying 
	\begin{equation}{\label{muupperldp}}
		\mu_\varepsilon(\{x\in\mathbb{R}^d:\rho\left(x, \Phi(s) \right) \geq \delta\}) \leq \exp (-\varepsilon^{-2}(s-\gamma))
	\end{equation}
	for any $\varepsilon \in (0, \varepsilon_0)$. Here, we set $\Phi(s)=\{x\in\mathbb{R} : \mathcal{S}(x)\leq s \}$.
	
	\textbf{$s=0$:} \eqref{muupperldp} is clearly true.
	
	\textbf{$s>0$:} We choose a $ \rho_1 \in (0,\delta) $ satisfied $ \max_{i\in \mathcal{L}} V(K_i,\partial G_i)  <\frac{\gamma}{4}$. By Lemma \ref{lem:invariantmeasureexist}, there exist proper $\rho_1,\rho_2,\beta,\varepsilon' \in (0,+\infty)$ such that we have $\max_{z\in\partial G_i}E_z^{\varepsilon}\tau_1^{\varepsilon} <\beta$ and 
	\begin{eqnarray}\label{muupperldp1}
		\mu_\varepsilon(\{x\in\mathbb{R}^d:\rho\left(x, \Phi(s) \right) \geq \delta\}) &=& \int_{\partial g} E^{\varepsilon}_y
		\int^{\tau^{\varepsilon}_1}_0 1_{\{x\in\mathbb{R}^d:\rho\left(x, \Phi(s) \right) \geq \delta\}} (X^{\varepsilon}_t) d t
		\nu_{\varepsilon} (d y) \nonumber\\
		&=& \sum_{i\in \mathcal{L}}\int_{\partial g_i} E^{\varepsilon}_y
		\int^{\tau^{\varepsilon}_1}_{\sigma^{\varepsilon}_1} 1_{\{x\in\mathbb{R}^d:\rho\left(x, \Phi(s) \right) \geq \delta\}} (X^{\varepsilon}_t) d t
		\nu_{\varepsilon} (dy) \nonumber\\
		& \leq & \sum_{i\in \mathcal{L}} \max_{z\in\partial G_i}P_z^{\varepsilon}\left(\tau_{\{x\in\mathbb{R}^d:\rho\left(x, \Phi(s) \right) \geq \delta\}}^{\varepsilon}\leq\tau_1^{\varepsilon}\right) \cdot \max_{z \in \partial G_i} E_z^{\varepsilon}\tau_1^{\varepsilon}\cdot \nu_{\varepsilon}(\partial g_i)
	\end{eqnarray}
	for any $\varepsilon\in(0,\varepsilon')$.
	
	For any $i \in \mathcal{L}$, according to the definition of $\mathcal{S}(x) $ and the proof of \eqref{invariantmeasureestimationofG_i5}, we have 
	$$V(K_i,x) \geq s + \min_{i \in \mathcal{L}} W(K_i)- W(K_i)$$
	and there exists an $\varepsilon '' \in (0, +\infty)$ satisfied
	\begin{equation}\label{muupperldp2}
		\max_{z\in\partial G_i}P_z^{\varepsilon}\left(\tau_{\{x\in\mathbb{R}^d:\rho\left(x, \Phi(s) \right) \geq \delta\}}^{\varepsilon}\leq\tau_1^{\varepsilon}\right) \leq \exp\left(-\varepsilon^{-2}\left(s + \min_{i \in \mathcal{L}} W(K_i)- W(K_i)-\frac{\gamma}{3}\right)\right),
	\end{equation}
	for any $\varepsilon \in (0,\varepsilon'')$.
	
	Hence, combining \eqref{muupperldp1} with \eqref{muupperldp2}, we have
	\begin{equation*}
		\mu_\varepsilon(\{x\in\mathbb{R}^d:\rho\left(x, \Phi(s) \right) \geq \delta\})\leq l \beta \exp\left(-\varepsilon^{-2}\left(s-\frac{2\gamma}{3}\right)\right),
	\end{equation*}
	for any $\varepsilon \in (0, \varepsilon ' \wedge \varepsilon '') $, which implies \eqref{muupperldp}.

\end{proof}

\begin{rem}\label{rem:estimateofS}
	According to the proof of \eqref{ash6} and Lemma \ref{lem:lowerboundedofqp}, for any $x \in B_0^c(M) $, we can give a lower bound of $\mathcal{S}(x)$:
	\begin{eqnarray*}
		\mathcal{S}(x) &\geq& \min_{i \in \mathcal{L}} W(K_i) + \min_{i \in \mathcal{L}} V(K_i, x) -\min_{i \in \mathcal{L}} W(K_i)\nonumber\\
		&\geq& \min_{i\in\mathcal{L},y \in \partial B_0(M) } V(K_i,y) + \min_{y \in \partial B_0(M) } V(y,x)\nonumber\\
		&\geq& \min_{i\in\mathcal{L},y \in \partial B_0(M) } V(K_i,y) + \frac{2\zeta\kappa}{{\bar{\lmd}}^2}(|x|-M).
	\end{eqnarray*} 
\end{rem}

\medskip

\section{Examples and Numerical Simulations}

To show our conclusion more intuitive, we will give some specific examples with theoretical analyses and
numerical simulations. These examples are organized from simple to complex. All
the numerical simulations are obtained by the tamed-Euler method provided in {\cite{ns}}.

\begin{example}\label{gsystem}
A Gradient System:
\end{example}

Let us consider the gradient system with Brownian perturbations in $\mathbb{R}^2$:
\[ \left\{ \begin{array}{cl}
     \left( \begin{array}{l}
       d X^{\varepsilon}_t\\
       d Y^{\varepsilon}_t
     \end{array} \right) &= - \left( \begin{array}{c}
       (X^{\varepsilon}_t)^3 - X^{\varepsilon}_t\\
       Y^{\varepsilon}_t
     \end{array} \right) d t + \varepsilon \left( \begin{array}{l}
       d W^{(1)}_t\\
       d W^{(2)}_t
     \end{array} \right),\\
     \left( \begin{array}{l}
       X^{\varepsilon}_0\\
       Y^{\varepsilon}_0
     \end{array} \right) &= \left( \begin{array}{l}
       x_0\\
       y_0
     \end{array} \right) .
   \end{array} \right. \]
Here we choose $(x_0, y_0)^T$ arbitrarily in $\mathbb{R}^2$ and $J (x, y) = \frac{x^4}{4} + \frac{y^2}{2} - \frac{x^2}{2} + 1$.  Figure \ref{fig:gradient13} is the phase graph of the corresponding deterministic equation.  

{\tmstrong{Theoretical analysis:}} This system has three equivalent sets: $K_1=(0,0),\ K_2=(-1,0),\ K_3=(1,0)$. According to Corollary \ref{rem:howtogetstable}, $K_2, K_3$ are stable. By Proposition \ref{prop:sstoos}, $K_1$ is unstable. 
In this special case, from the stationary Fokker-Planck equation of this example, we have
\[ \mu_{\varepsilon}(dxdy) = \left( \int_{\mathbb{R}^2} \exp \left( - \frac{2 J (x,
   y)}{\varepsilon^2} \right) d x d y \right)^{- 1} \exp \left( - \frac{2 J
   (x, y)}{\varepsilon^2} \right)dxdy. \]
By using the above explicit solution of this example, it is not difficult to check
$\mu = \frac{1}{2} \delta_{(- 1, 0)} + \frac{1}{2} \delta_{(1, 0)}$.

The numerical simulation of this system supports the theoretical analysis. 

{\tmstrong{Numerical Simulation:}} Figure \ref{fig:gradientsys} shows the distribution of the numerical solution with $(x_0, y_0) =
(0, 0)$, $\varepsilon = 0.01$ and time $T = 10000$. In this experiment, the
numerical solution is obtained by choosing step size $h = 0.01$ with $100$ samples.

\begin{center}
\tmfloat{h}{small}{figure}{\raisebox{0.0\height}{\includegraphics[width=5.2cm,height=4.0cm]{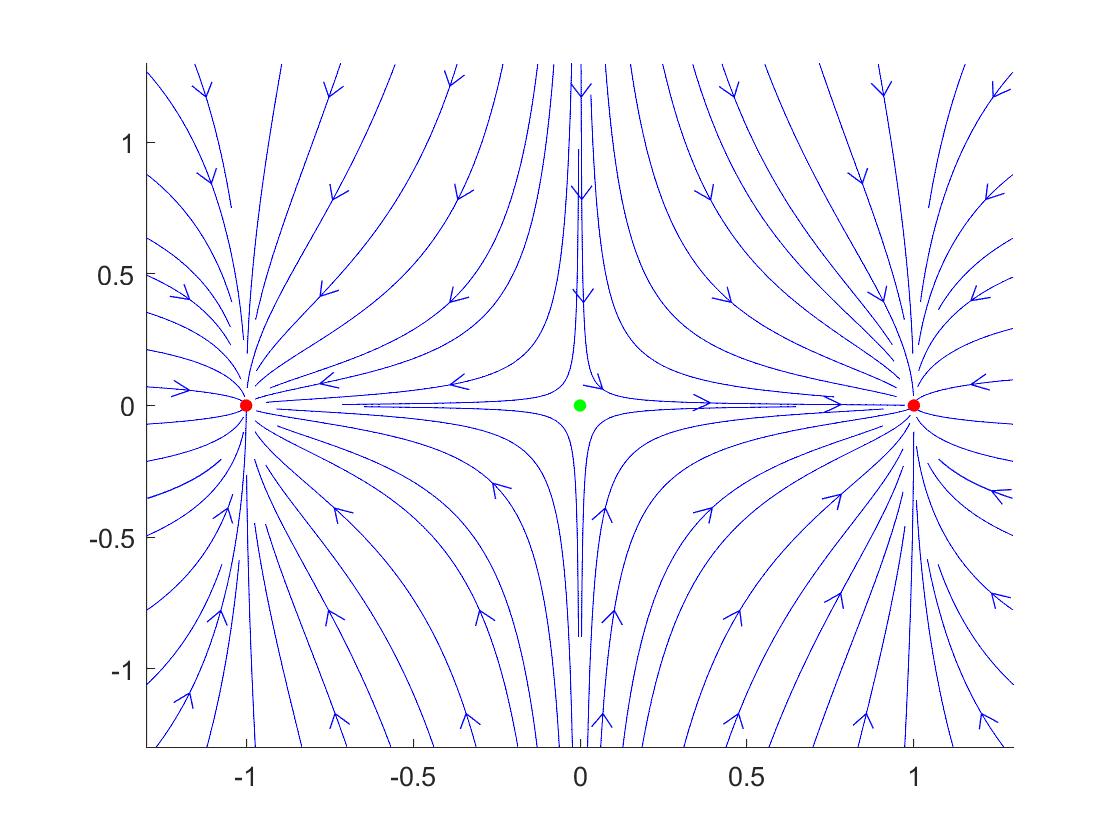}}}{\label{fig:gradient13}}
\tmfloat{h}{small}{figure}{\raisebox{0.0\height}{\includegraphics[width=5.2cm,height=4.0cm]{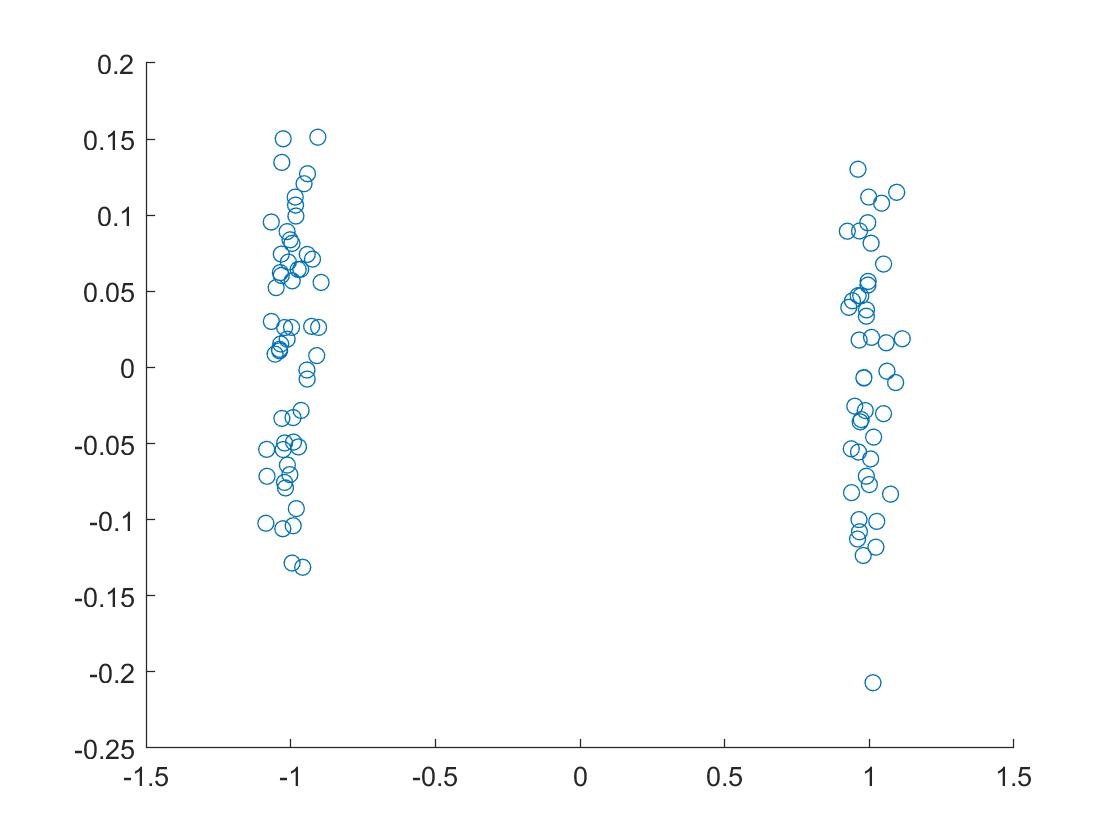}}}{\label{fig:gradientsys}}
\end{center}

\begin{rem}\label{rem:difficultreason}
	In this example, the explicit solution tells us that the $\min W(K_i)$ must be attached at $K_2, K_3$. But in general, it seems very difficult to show which equivalent sets can attach the $\min W(K_i)$, because it is not easy to give a proper upper estimate of $V(K_i, K_j)$.  
\end{rem}

\bigskip

The following example comes from \cite{chen}, we will explain this example in our view.
\begin{example}\label{Bernoulli}
	Stochastic Bernoulli Equation:
\end{example}
Let us set 
$$ 
O(x,y)=(x^2+y^2)^2-4(x^2-y^2), \quad U(O)=\frac{O^2}{2(1+O^2)^{\frac{3}{4}}}, \quad   \Theta(O)=\frac{O}{(1+O^2)^{\frac{3}{8}}}.
$$
It can be checked that the following SDE:
\[ \left\{ \begin{array}{cl}
	\left( \begin{array}{l}
		d X^{\varepsilon}_t\\
		d Y^{\varepsilon}_t
	\end{array} \right) &= \left( - \left( \begin{array}{c}
		\partial_x U (X^{\varepsilon}_t, Y^{\varepsilon}_t)\\
		\partial_y U (X^{\varepsilon}_t, Y^{\varepsilon}_t)
	\end{array} \right) + \left( \begin{array}{c}
	 \partial_y \Theta(X^{\varepsilon}_t, Y^{\varepsilon}_t)\\
	-\partial_x \Theta (X^{\varepsilon}_t, Y^{\varepsilon}_t)
	\end{array} \right) \right) d t + \varepsilon \left( \begin{array}{l}
		d W^{(1)}_t\\
		d W^{(2)}_t
	\end{array} \right),\\
	\left( \begin{array}{l}
		X^{\varepsilon}_0\\
		Y^{\varepsilon}_0
	\end{array} \right) &= \left( \begin{array}{l}
		x_0\\
		y_0
	\end{array} \right) .
\end{array} \right. \]
satisfies the condition of Corollary \ref{co:qgsystem}. Here we choose $(x_0, y_0)^T$ arbitrarily in $\mathbb{R}^2$. Figure \ref{fig:deterbernoulli} is the phase graph for the corresponding deterministic equation.
\begin{center}
	\tmfloat{h}{small}{figure}{\raisebox{0.0\height}{\includegraphics[width=5.2cm,height=4.0cm]{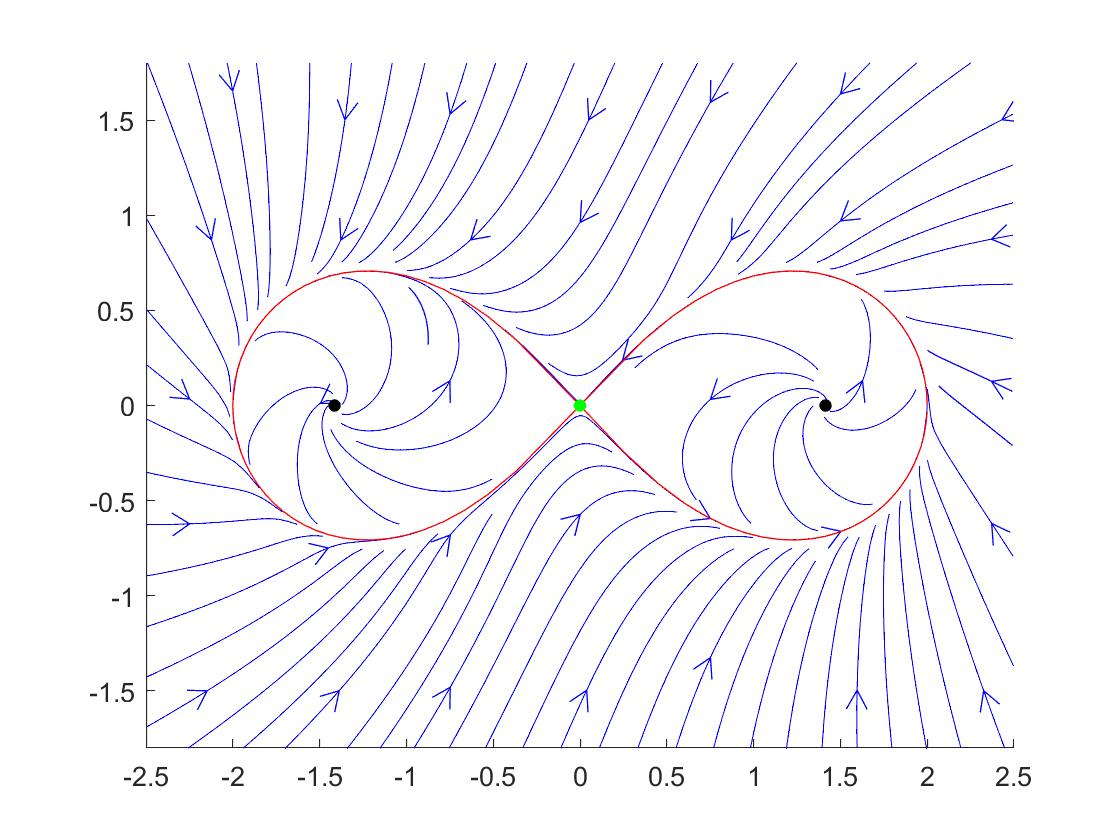}}}{\label{fig:deterbernoulli}}
	\tmfloat{h}{small}{figure}{\raisebox{0.0\height}{\includegraphics[width=5.2cm,height=4.0cm]{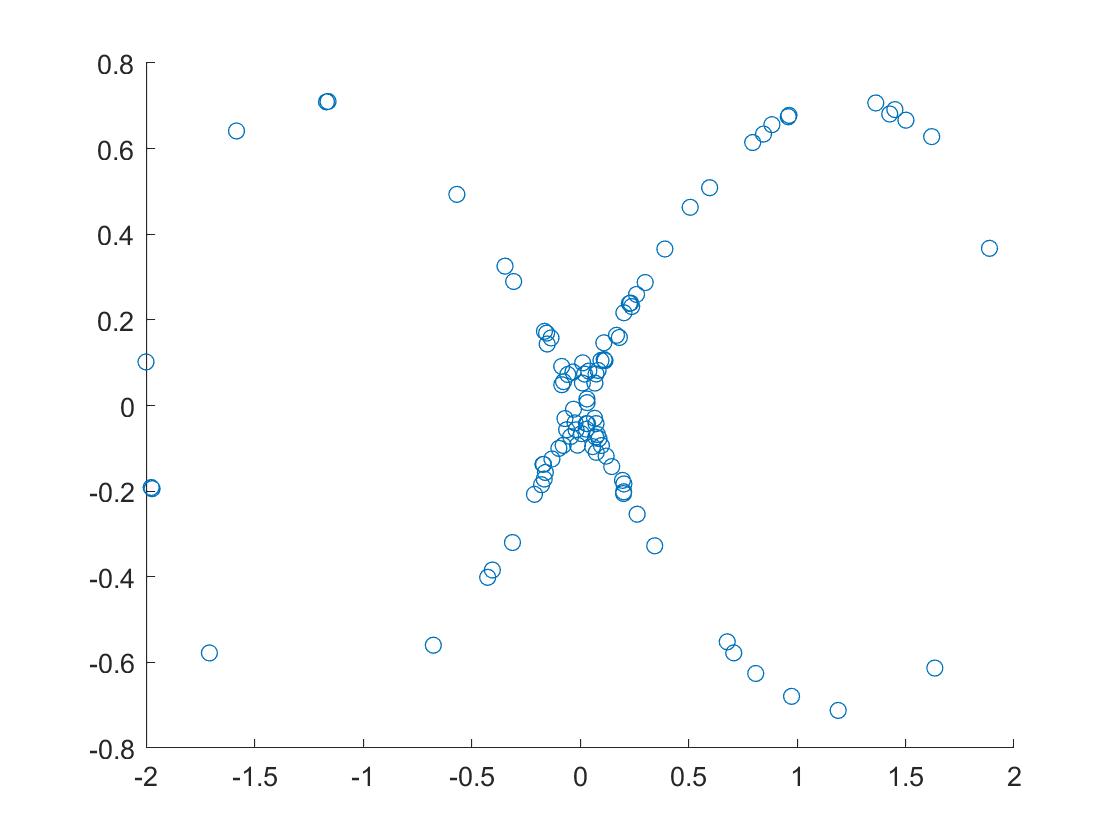}}}{\label{fig:stobernoulli}}
\end{center}

\bigskip

{\tmstrong{Theoretical analysis:}}  This system has three equivalent set $ K_1=\{(x,y)\in\mathbb{R}^2:(x^2+y^2)^2 = 4(x^2-y^2) \} $, $K_2=(-\sqrt2,0)$ and $K_3=(\sqrt2,0)$. According to Corollary \ref{rem:howtogetstable}, $K_1$ is stable. By Proposition \ref{prop:sstoos}, $K_2, K_3$ are unstable.  In Example \ref{Bernoulli}, $K_1$ is the only stable set, thus $\mu$ can only supports on $K_1$. Furthermore, according to Corollary \ref{co:cap} $\mu$ can only supports on $(0,0)$. 

The following numerical simulation supports the theoretical analysis. 

{\tmstrong{Numerical Simulation:}} 
Figure \ref{fig:stobernoulli} shows the distribution of the numerical solution with
$(x_0, y_0) = (0, 0.5)$, $\varepsilon = 0.0003$ and time $T = 30000$. In this
experiment, the numerical solution is obtained by choosing step size $h =
0.01$ with $100$ samples. The sample points are concentrated near $(0,0)$.

\bigskip

\begin{example}\label{duffing}
	Stochastic Duffing Equation:
\end{example}

Let us consider the Duffing equation with Brownian perturbations in 
$\mathbb{R}^2$:
\begin{equation}\label{stoduffing}
  \left\{ \begin{array}{cll}
    \left( \begin{array}{l}
      d X^{\varepsilon}_t\\
      d Y^{\varepsilon}_t
    \end{array} \right) &=& \left( - \left( \begin{array}{c}
      (X^{\varepsilon}_t)^3 - X^{\varepsilon}_t\\
      Y^{\varepsilon}_t
    \end{array} \right) + \left( \begin{array}{c}
      - Y^{\varepsilon}_t\\
      (X^{\varepsilon}_t)^3 - X^{\varepsilon}_t
    \end{array} \right) \right) d t + \varepsilon \left( \begin{array}{l}
      d W^{(1)}_t\\
      d W^{(2)}_t
    \end{array} \right),\\
    \left( \begin{array}{l}
      X^{\varepsilon}_0\\
      Y^{\varepsilon}_0
    \end{array} \right) &=& \left( \begin{array}{l}
      x_0\\
      y_0
    \end{array} \right) .
  \end{array} \right. 
\end{equation}
Here we choose $(x_0, y_0)^T$ arbitrarily in $\mathbb{R}^2$ and $J (x, y)
= \frac{x^4}{4} + \frac{y^2}{2} - \frac{x^2}{2} + 1$.  Figure \ref{fig:duffing13} is the phase graph of the Duffing equation. It can be checked that \eqref{stoduffing}  satisfying the requirements in Corollary \ref{co:qgsystem}.

{\tmstrong{Theoretical analysis:}} We denote $K_1 = (0, 0)$, $K_2 = (- 1, 0)$,  $K_3 = (1, 0)$. According to Corollary \ref{rem:howtogetstable}, $K_2, K_3$ are stable. By Proposition \ref{prop:sstoos}, $K_1$ is unstable.
According to Corollary \ref{co:qgsystem}, 
the weak limitation of the invariant measure of \eqref{stoduffing} does
not support on the saddle point $K_1$, which is unstable. Owing to the symmetrical property of \eqref{stoduffing}, we can prove $W(K_2)=W(K_3)$, which implies our method can not further exclude $K_2$ or $K_3$.

 To show $W (K_2) = W
(K_3)$, we need to consider the value of $V (K_2, K_1)$, $V (K_2, K_3)$, $V
(K_1, K_2)$, $V (K_1, K_3)$, $V (K_3, K_2)$ and $V (K_3, K_1)$. We claim that
$$V ((x_1, x_2), (y_1, y_2)) = V ((- x_1, -x_2), (- y_1, - y_2)), \forall (x_1, x_2), (y_1, y_2) \in \mathbb{R}^2.$$
It is because for any $(\varphi^{(1)},
\varphi^{(2)}) \in A C ([0, T] ; \mathbb{R}^2)$ satisfying $(\varphi_0^{(1)},
\varphi_0^{(2)}) = (x_1, x_2)$ and $(\varphi_T^{(1)}, \varphi_T^{(2)}) \\= (y_1, y_2)$, there
is a $(\phi^{(1)}, \phi^{(2)}) = (- \varphi^{(1)}, - \varphi^{(2)}) \in A C ([0, T] ;
\mathbb{R}^2)$ satisfying $(\phi_0^{(1)}, \phi_0^{(2)}) = (- x_1, - x_2)$, $(\phi_T^{(1)},
\phi_T^{(2)}) = (- y_1, - y_2)$ and
\begin{eqnarray*}
  S_{0 T} (\phi) & = & \frac{1}{2} \int^T_0 (\dot{\phi}_t^{(1)} - b^{(1)} (\phi_t))^2
  + (\dot{\phi}_t^{(2)} - b^{(2)} (\phi_t))^2 d t\\
  & = & \frac{1}{2} \int^T_0 (- \dot{\varphi}_t^{(1)} - b^{(1)} (- \varphi_t))^2 + (-
  \dot{\varphi}_t^{(2)} - b^{(2)} (- \varphi_t))^2 d t\\
  & = & \frac{1}{2} \int^T_0 (\dot{\varphi}_t^{(1)} - b^{(1)} (\varphi_t))^2 +
  (\dot{\varphi}_t^{(2)} - b^{(2)} (\varphi_t))^2 d t\\
  & = & S_{0 T} (\varphi) .
\end{eqnarray*}
Thus, we have $V (K_2, K_1) = V (K_3, K_1)$, $V (K_2, K_3) = V (K_3, K_2)$ and
$V (K_1, K_2) = V (K_1, K_3)$. From above facts we get $W (K_2) = W (K_3)$, which implies that $\mu$ may support both on $(- 1, 0)$ and $(1,
0)$.

The following results of numerical simulation supports our theoretical analysis above.

\begin{center}
\tmfloat{h}{small}{figure}{\raisebox{0.0\height}{\includegraphics[width=5.2cm,height=4.0cm]{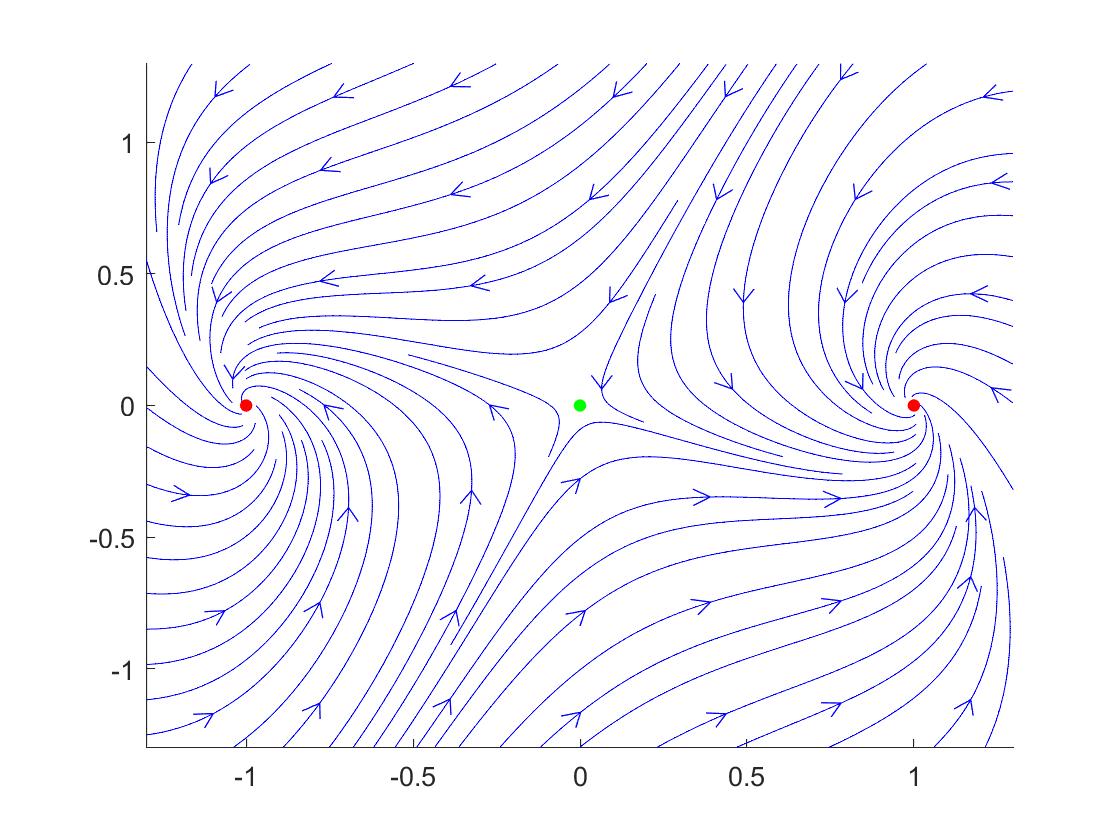}}}{\label{fig:duffing13}}
\tmfloat{h}{small}{figure}{\raisebox{0.0\height}{\includegraphics[width=5.2cm,height=4.0cm]{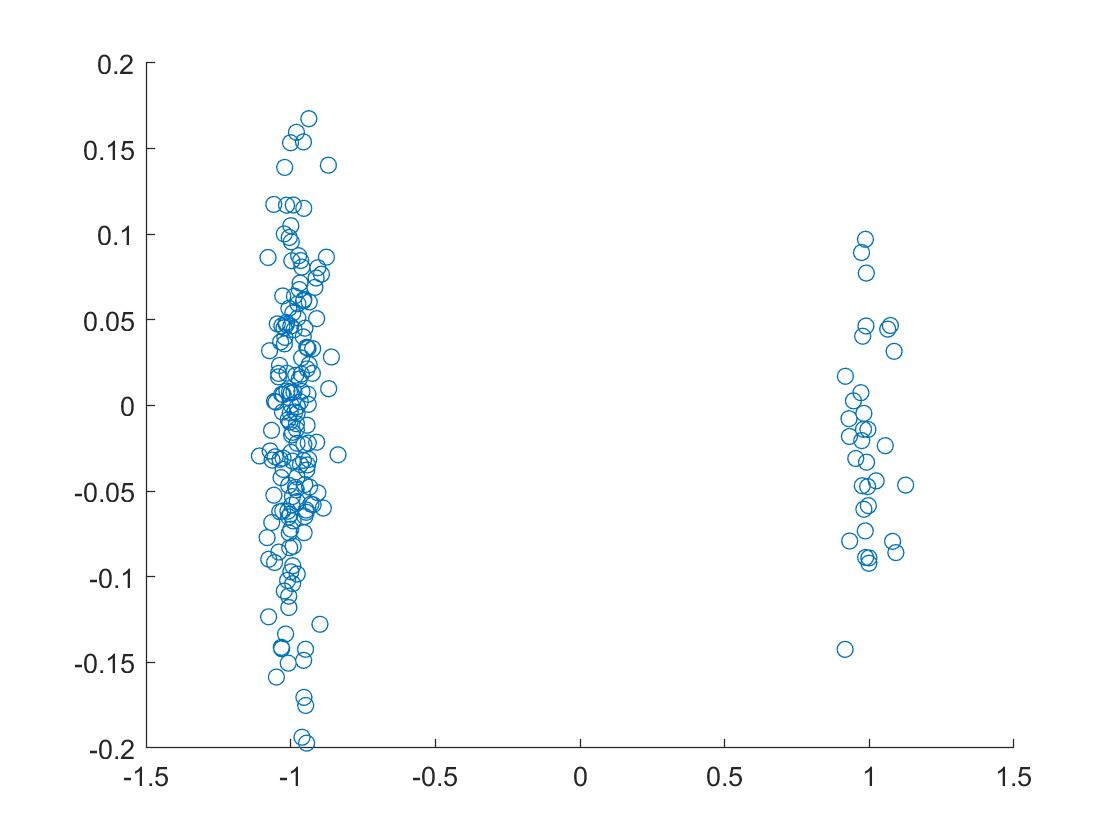}}}
{\label{fig:duffinglongsto}}
\end{center}

{\tmstrong{Numerical Simulation:}} Figure \ref{fig:duffinglongsto} shows the distribution of
numerical solution with $(x_0, y_0) = (- 0.05, 0.05)$, $\varepsilon = 0.01$
and time $T = 20000$. In this experiment, the numerical solution is obtained
by choosing step size $h = 0.01$ with $200$ samples. According to 
Figure \ref{fig:duffinglongsto},
we can see that although the initial point $(- 0.05,
0.05)$ is in the attracting domain of $(- 1, 0)$, the solution of \eqref{stoduffing} with small perturbations
beginning at $(- 0.05, 0.05)$ still concentrates on both $(- 1, 0)$ and $(1, 0)$. This fact is quite different from the deterministic case.

\begin{example}\label{Non-symmetrical Example}
	Non-symmetrical Equation:
\end{example}

Let us set $J (x, y) = x^6 + y^6 + 3 x^4 y^2 + 3 y^4 x^2 - 1.515 (x^4 + y^4 + 2 x^2
y^2) + 0.03 (x^2 + y^2) + 1$. We denote $\nabla J (x, y) = (J_1 (x, y),
J_2 (x, y))^T$ and $H (x, y) = (- J_2 (x, y), J_1 (x, y))^T$. It can be
checked the following SDE:
\[ \left\{ \begin{array}{cl}
	\left( \begin{array}{l}
		d X^{\varepsilon}_t\\
		d Y^{\varepsilon}_t
	\end{array} \right) &= \left( - \left( \begin{array}{c}
		J_1 (X^{\varepsilon}_t, Y^{\varepsilon}_t)\\
		J_2 (X^{\varepsilon}_t, Y^{\varepsilon}_t)
	\end{array} \right) + \left( \begin{array}{c}
		- J_2 (X^{\varepsilon}_t, Y^{\varepsilon}_t)\\
		J_1 (X^{\varepsilon}_t, Y^{\varepsilon}_t)
	\end{array} \right) \right) d t + \varepsilon \left( \begin{array}{l}
		d W^{(1)}_t\\
		d W^{(2)}_t
	\end{array} \right),\\
	\left( \begin{array}{l}
		X^{\varepsilon}_0\\
		Y^{\varepsilon}_0
	\end{array} \right) &= \left( \begin{array}{l}
		x_0\\
		y_0
	\end{array} \right) .
\end{array} \right. \]
satisfying the requirements in Corollary \ref{co:qgsystem}. Here we choose $(x_0, y_0)^T$ arbitrarily in $\mathbb{R}^2$. Figure \ref{fig:nonsym105} is the phase graph for the corresponding deterministic equation.

{\tmstrong{Theoretical analysis:}} In Example \ref{Non-symmetrical Example}, the non-symmetrical property is to say that the stable sets of this system are not symmetrical. As a comparison, Example \ref{duffing} is a symmetrical system. For this system, it has three equivalent sets $K_1=(0,0)$, $K_2=\{(x,y):x^2+y^2 = 0.01\}$ and $K_3=\{(x, y) : x^2 + y^2 = 1 \}$.
$K_2$ and $K_3$ are equivalent sets because of
$$\nabla J (x, y)=0, H (x, y)=0, \ \forall \ (x,y) \in K_2,K_3.$$
Precisely speaking, for any $(x_1,y_1), (x_2,y_2) \in K_3$, denote the length of arc between  $(x_1,y_1)$ and$(x_2,y_2)$ on $\{(x, y) : x^2 + y^2 = 1 \} $ as $s$. We can choose $\varphi \in C([0,\epsilon^{-1}s];K_3)$ connecting $(x_1,y_1)$ and $ (x_2,y_2)$ satisfying
\begin{equation*}
	\left\{ \begin{array}{l}
		 \begin{array}{l}
			 \dot{\varphi}^{(1)}_t=\\
			 \dot{\varphi}^{(2)}_t=
		\end{array}   \begin{array}{l}
			\epsilon{\varphi}^{(2)}_t,\\
			-\epsilon{\varphi}^{(1)}_t.
		\end{array}\\
	\end{array} \right. 
\end{equation*}
Thus, we have
\begin{eqnarray*}
	V((x_1,y_1), (x_2,y_2))& \leq &S_{0\epsilon^{-1}}(\varphi)\\
	& = & \frac{1}{2} \int^{\epsilon^{-1}s}_0 \epsilon^2 (({\varphi}^{(1)})^2+({\varphi}^{(2)})^2) d t\\
	& = & \frac{\epsilon s}{2}.
\end{eqnarray*}
By the arbitrariness of choosing $\varepsilon$, we have $V((x_1,y_1), (x_2,y_2))=0 $. Hence, $K_2, K_3$ are equivalent sets. According to Corollary \ref{rem:howtogetstable}, $K_1, K_3$ are stable. By Proposition \ref{prop:sstoos}, $K_2$ is unstable. Furthermore, we can check that $K_3$ is only one stable set getting the $\min W(K_i)$.

In fact, it is easy to check $\tilde{V}(K_2, K_1) = \tilde{V}(K_2, K_3)=0$ and $\tilde{V}(K_1,K_3)=\tilde{V}(K_3,K_1)=+\infty$. By Lemma \ref{lem:lowerboundedofqp}, we get $\tilde{V}(K_3, K_2) \geq 0.5285$. As for $\tilde{V}(K_1, K_2)$, we choose
\begin{equation*}
	\left\{
	\begin{array}{ll}
		\varphi_t^{(1)}& = t,\\
		\varphi_t^{(2)}& = 0,
	\end{array}  t \in \left[0,0.01\right] \right .
\end{equation*}
to connect $K_1$ with $K_2$. Setting $T=0.01$, it is not difficult to calculate $S_{0T}(\varphi) \approx 5.03 \times 10^{-3} $, which implies $ \tilde{V}(K_3, K_2) > S_{0T}(\varphi) \geq\tilde{V}(K_1, K_2) $. Thus, using the definition of $W(K_i)$, we have $W(K_3)<W(K_1)$, which implies $\mu$ only supports on $K_3=\{(x, y) : x^2 + y^2 = 1 \}$.

The following numerical simulation supports the theoretical analysis. 

{\tmstrong{Numerical Simulation:}} 
Figure \ref{fig:nonsymsys} shows the distribution of the numerical solution with
$(x_0, y_0) = (0, 0)$, $\varepsilon = 0.001$ and time $T = 10000$. In this
experiment, the numerical solution is obtained by choosing step size $h =
0.01$ with $100$ samples. It is interesting to notice that although $(0, 0)$
is a stable point of this system, in a long time observation, we may only
see the state $\{ (x, y) : x^2 + y^2 = 1 \}$.

\begin{center}
	\tmfloat{h}{small}{figure}{\raisebox{0.0\height}{\includegraphics[width=5.2cm,height=4.0cm]{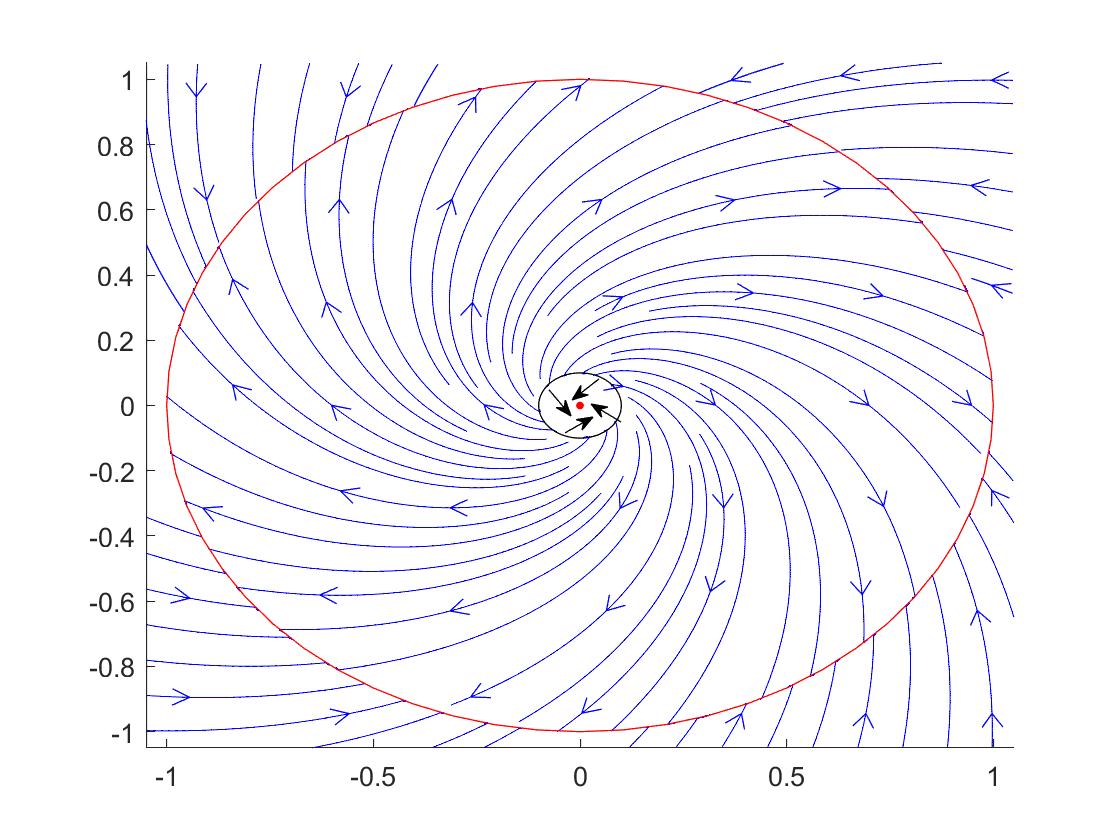}}}{\label{fig:nonsym105}}
	\tmfloat{h}{small}{figure}{\raisebox{0.0\height}{\includegraphics[width=5.2cm,height=4.0cm]{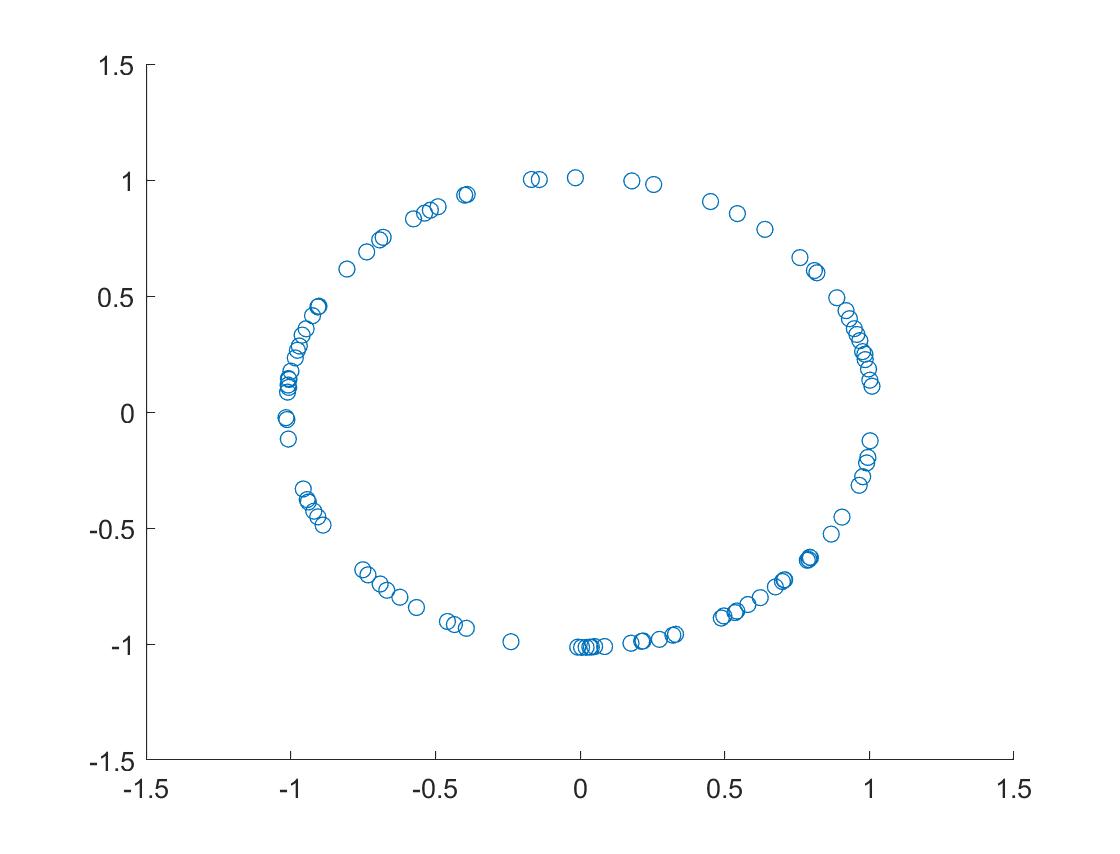}}}{\label{fig:nonsymsys}}
\end{center}

\section*{Acknowledgements}

The authors are grateful to the helpful discussion with
Jifa Jiang and Lifeng Chen. The authors would like to thank Derui Sheng
for the help in numerical simulation. Zhao Dong was partially supported by National Key R\&D Program of China (No. 2020YFA0712700), Key Laboratory of Random Complex Structures and Data Science, Academy of Mathematics and Systems Science, Chinese Academy of Sciences (No. 2008DP173182), NSFC No.11931004, NSFC No.12090014. Liang Li was partially supported by NSFC NO.11901026, NSFC NO. 12071433, NSFC NO.12171032.

\end{document}